\documentclass[10pt]{elsarticle}

\usepackage{amsthm,amsmath,amsfonts,amssymb,amscd,mathrsfs}
\usepackage{txfonts}
\usepackage{supertabular,soul}
\usepackage[usenames,dvipsnames]{xcolor}
\usepackage{tikz, graphicx,color,geometry}
 \usepackage{multirow}
 \usetikzlibrary{arrows}
\usepackage[pdftex,
            pdfauthor={Reber},
            pdftitle={Time-Varying Time Delays and Intrinsic Stability},
            pdfsubject={Switched Systems, Network Theory, Nonlinear Dynamical Systems},
            pdfkeywords={Switched Systems, Network Theory, Nonlinear Dynamical Systems}]{hyperref}

\usepackage{bbm} 
\usepackage{hyperref}
\usepackage{yfonts}
\usepackage{eucal}
\usepackage{overpic}
\usetikzlibrary{calc}
\usepackage{enumitem}
\usepackage{dsfont}
\usepackage{commath}
\usepackage{xcolor}
\usepackage{soul}

\newcommand{\comment}[1]{}

\newtheorem{result}{Main Result}
\newtheorem{theorem}{Theorem}

\newtheorem{proposition}{Proposition}
\newtheorem{definition}{Definition}
\newtheorem{thm}{Theorem}[section]

\newtheorem{lemma}{Lemma}

\newtheorem{example}[thm]{Example}

\theoremstyle{remark}

\providecommand*{\propertyautorefname}{Property}

\let\oldmarginpar\marginpar
\renewcommand\marginpar[1]{\oldmarginpar[\raggedleft\footnotesize #1]%
{\raggedright\footnotesize #1}}

\begin{document}
\begin{frontmatter}

\date{\today}

\title{Intrinsic Stability: Global Stability of Dynamical Networks and Switched Systems Resilient to any Type of Time-Delays}
\author[david]{David Reber}
\author[ben]{Benjamin Webb}
\address[david]{Department of Mathematics, Brigham Young University, Provo, UT 84602, USA, davidpreber@mathematics.byu.edu}
\address[ben]{Department of Mathematics, Brigham Young University, Provo, UT 84602, USA, bwebb@mathematics.byu.edu}


\begin{abstract}
In real-world networks the interactions between network elements are inherently time-delayed. These time-delays can not only slow the network but can have a destabilizing effect on the network's dynamics leading to poor performance. The same is true in computational networks used for machine learning etc. where time-delays increase the network's memory but can degrade the network's ability to be trained. However, not all networks can be destabilized by time-delays. Previously, it has been shown that if a network or high-dimensional dynamical system is intrinsically stabile, which is a stronger form of the standard notion of global stability, then it maintains its stability when constant time-delays are introduced into the system. Here we show that intrinsically stable systems, including intrinsically stable networks and a broad class of \emph{switched systems}, i.e. systems whose mapping is time-dependent, remain stable in the presence of any type of time-varying time-delays whether these delays are periodic, stochastic, or otherwise. We apply these results to a number of well-studied systems to demonstrate that the notion of intrinsic stability is both computationally inexpensive, relative to other methods, and can be used to improve on some of the best known stability results. We also show that the asymptotic state of an intrinsically stable switched system is exponentially independent of the system's initial conditions.
\end{abstract}


\end{frontmatter}

\section{Introduction}

The study of networks deals with understanding the properties of systems of interacting elements. In the social sciences these elements are typically individuals whose social interactions create networks such as Facebook and Twitter. In the biological sciences networks range from metabolic networks of single-cell organisms to the neuronal networks of the brain to the larger physiological network of the organs within the body. Networks such as citation networks and other well-studied networks such as the World Wide Web belong to what are referred to as information networks. In the technological sciences examples of networks include the internet, power grids, and transportation networks. (For an overview of these different types of networks see \cite{networks}.)

These real-world networks are dynamic in that both their \emph{topology}, which is the network's structure of interactions, and the state of the network are time dependent. Here our focus is on the changing state of the network, which is often referred to as the \emph{dynamics on} the network. The \emph{state} of the network is the collective states of the network elements. The fundamental concept in a dynamical network is that the state of a given element depends on the dynamics of its \emph{neighbors}, which are the other network elements that directly interact with this element.

Here we refer to the emergent behavior of these interacting elements, which is the changing state of these network elements, as the \emph{network's dynamics}. The network's dynamics can be periodic, as is found in many biological networks \cite{PeriodicBiologyDelays}, synchronizing, which is the desired condition for transmitting power over large distance in power grids \cite{SyncPowerGrid}, and stable or multistable dynamics such as is found in gene-regulatory networks \cite{MultistableGenes}, etc.

In real-world networks the interactions between network elements are inherently time-delayed. This comes from the fact that network elements are spatially separated, that information and other quantities can only be processed and transmitted at finite speeds, and that these quantities can be slowed by network traffic \cite{IntroPaper}. These time-delays not only slow the network, leading to poor performance, but can have a destabilizing effect on the network's dynamics, which can lead to network failure \cite{destabilizing1,destabilizing2}. The same is true in computational networks used for machine learning etc. where time-delays can be used to increase the network's ability to detect long term temporal dependencies \cite{memory} but at the cost of potentially degrading the ability to train the network.

Not all networks can be destabilized by time-delays. In \cite{BunWebb0} the notion of intrinsic stability was introduced, which is a stronger form of the standard notion of \emph{stability}, i.e. a system in which there is a globally attracting fixed point (see \cite{BunWebb2} for more details). If a network is intrinsically stable the authors showed that \emph{constant-type time-delays}, which are delays that do not vary in time, have no effect on the network's stability. As was shown in \cite{BunWebb0} an advantage of the intrinsic stability method over other methods such as Lyapunov-type methods, Linear Matrix Inequalities (LMI), and Semi-Definite Programming (SPD) is that determining whether a network is intrinsically stable can be done with respect to the lower-dimensional undelayed network and does not require the creation of special functions or the use of interior point methods, etc. What is required is finding the spectral radius of the network's Lipschitz matrix, which for large network's can be done efficiently by use of the power method \cite{powermethod}.

In many situations, however, the delays networks or, or more generally high-dimensional systems, experience are not constant-type time-delays. Time delays can be periodic, such as the daily/annual cycles in population models \cite{PeriodicPopulation}, or even stochastic as in traffic models \cite{StochasticTrains}. It is worth noting that these time-varying time-delays are more complicated than constant-type time delays and as such the theory of systems with time-varying time-delays is less developed than the theory of systems with constant time-delays, which in turn is less developed than the theory of systems without delays.

The main goal of this paper is to further develop the theory describing the stability of dynamical systems that experience time delays. Here we show that intrinsically stable systems, including intrinsically stable networks and a broad class of \emph{switched systems}, i.e. systems whose mapping is time-dependent, remain stable in the presence of any time-varying time-delays whether these delays are periodic, stochastic, or otherwise (see Main Results \ref{TimeVaryingIntrinsic} and \ref{RIIntrinsic}). We apply these results to a number of well-studied systems to demonstrate that the notion of intrinsic stability is both computationally inexpensive, relative to other methods, and can be used to improve on some of the best known results. (See for instance Example \ref{IntStable4.1}, compared to results found in \cite{dependent1,dependent11,dependent12,dependent13,dependent14}.) We also show that the asymptotic state of intrinsically stable switched systems is independent of the system's initial conditions (see Main Result \ref{LimitOrbits}). This allows us to show that the globally attracting state of any intrinsically stable network and any time-delayed version of the network are identical (see Proposition \ref{FixedPoint}).

The main results of this paper are demonstrated using Cohen-Grossberg Neural (CGN) Networks, which are dynamical networks whose stability is often studied in the presence of constant-type and time-varying  time-delays \cite{neural3,neural5}. It is worth noting that the results of this paper justify the modeling of dynamical networks and switched systems without formally including delays in the model if it is known that the system is intrinsically stable. The reason is that although delays do change the specific details the system's dynamics, if the system is intrinsically stable it will have the same qualitative dynamics and same asymptotic state whether or not its delays are included (see Main Results \ref{LimitOrbits}, \ref{TimeVaryingIntrinsic}, and \ref{RIIntrinsic})). The advantage is that determining or correctly anticipating what delays a network will experience can be quite complicated for most any real-world system. Hence, this theory also has implications to system design as an intrinsically stable network will not be destabilized by any type of unexpected delays.

As this theory is in many ways different from the standard theory of stable dynamical systems, much effort has gone into, first, making this theory understandable and, second, emphasizing the computational and algorithmic advantages of this method. To this end examples are given throughout this paper describing each result, its usefulness, and how these results can be implemented in an efficient manner.

This paper is organized as follows. In Section \ref{sec:2} we define a dynamical network and the notion of intrinsic stability. In Section \ref{sec:3} we introduce networks with constant time-delays and show that these delayed systems have the same fixed points as their undelayed versions. In Section \ref{sec:4} we define switched networks, i.e. switched systems, and give our Main Results \ref{LimitOrbits} and \ref{TimeVaryingIntrinsic}, which show that intrinsically stable switched systems have the same asymptotic state irrespective of initial condition if intrinsically stable and intrinsically stable networks with time-varying time-delays have a globally attracting fixed point, respectively. We then apply this theory to linear systems with distinct delayed and undelayed interactions, which allow us to compare our results to some of the most well-studied time-delayed systems. In Section \ref{sec:5} we extend our results to a more general class of switched networks and similarly compare our results to a number of well-studied switching systems. In Section \ref{sec:6} we introduce some analytical and computational considerations related to determining whether a network is intrinsically stable. Section \ref{sec:8} contains some remarks about future work and applications of this theory. The Appendix contains the proofs of our results.

\section{Dynamical Networks}\label{sec:2}
A network is composed of a set of \emph{elements}, which are the individual units that make up the network and a collection of interactions between these elements. An \emph{interaction} between two network elements can be thought of as an element's ability to influence the behavior of the other network element. More generally, there is a \emph{directed interaction} from the $j^{th}$ to the $i^{th}$ elements of a network if the $j^{th}$ network element can influence the state of the $i^{th}$ network element (where there may be no influence of the $i^{th}$ network element on the $j^{th}$). The dynamics of a network can be formalized as follows:

\begin{definition}{\textbf{\emph{(Dynamical Network)}}}
Let $(X_i,d_i)$ be a complete metric space for $1\le i\le n$.
Let $(X,d_{max})$ be the complete metric space formed by endowing the product space $X=\bigoplus_{i=1}^n X_i$ with the metric
\[
d_{max}(\mathbf{x},\mathbf{y}) = \max_i d_i(x_i,y_i) \quad \text{ where } \mathbf{x},\mathbf{y} \in X \text{ and } \quad x_i,y_i \in X_i.
\]
Let $F:X \to X$  be a continuous map, with $i^{th}$ component function $F_i:X\to X_i$ given by
\[
F_i = F_i(x_1,x_2,\hdots,x_n) \quad \text{in which } \quad x_j \in X_j \quad \text{for } j=1,2,\hdots,n
\]
where it is understood that there may be no actual dependance of $F_i$ on $x_j$.
The dynamical system $(F,X)$ generated by iterating the function $F$ on $X$ is called a \emph{dynamical network}.
If an initial condition $\mathbf{x}^0\in X$ is given, we define the $k^{th}$ \emph{iterate} of $\mathbf{x}^0$ as $\mathbf{x}^k=F^k(\mathbf{x}^0)$, with orbit $\{F^k(\mathbf{x}^0)\}_{k=0}^\infty=\{\mathbf{x}^0,\mathbf{x}^1,\mathbf{x}^2,\hdots\}$ in which $\mathbf{x}^k$ is the state of the network at time $k \ge 0$.
\end{definition}

The component function $F_i = F_i(x_1,x_2,\hdots,x_n)$ describes the dynamics and interactions with the $i^{th}$ element of the network, where there is a directed interaction between the $i^{th}$ and $j^{th}$ elements if $F_i$ actually depends on $x_j$. The function $F:X\rightarrow X$ describes all interactions of the network $(F,X)$.
For the initial condition $\mathbf{x}^0\in X$ the state of the $i^{th}$ element at time $k\ge 0$ is $x^k_i=(F^k(\mathbf{x}^0))_i\in X_i$ so that $X_i$ is the element's state space. The state space $X=\bigoplus_{i=1}^n X_i$ is the collective state space of all network elements.

While the definition of a dynamical network allows for the function $F$ to be defined on general products of complete metric spaces, for the sake of intuition and direct applications of the theory we develop in this paper, our examples will focus on dynamical networks in which $X=\mathbb{R}^n$ with the infinity norm $\left\Vert\mathbf{x}\right\Vert_\infty=\max_i\left\lvert x_i\right\rvert$.
To give a concrete example of a dynamical network and to illustrate results throughout this paper, we will use Cohen-Grossberg neural (CGN) networks.

\begin{example}{(Cohen-Grossberg Neural Networks)}\label{DynamicalNetworkExample}
For $W\in \mathbb{R}^{n\times n}$, $\sigma : \mathbb{R}\to \mathbb{R}$, and $c_i,\epsilon\in\mathbb{R}$ for $1\le i \le n$ let $(C,\mathbb{R}^n)$ be the dynamical network with components
\begin{equation}
C_i(\mathbf{x}) = (1-\epsilon)x_i+\sum_{j=1}^n W_{ij}\sigma(x_j)+c_i
\quad 1\le i\le n,
\end{equation}
which is a special case of a Cohen-Grossberg neural network in discrete-time \cite{neural1}.
The function $\sigma$ is assumed to be bounded, differentiable, and monotonically increasing, with Lipschitz constant $K$, that is, \[\left\lvert \sigma(x)-\sigma(y) \right\rvert\le K\left\lvert x-y\right\rvert\]
for all $x,y\in\mathbb{R}$.
\end{example}

In a Cohen-Grossberg neural network the variable $x_i$ represents the \emph{activation} of the $i^{th}$ neuron.
The function $\sigma$ is a bounded monotonically increasing function, which describes the $i^{th}$ neuron's response to inputs.
The matrix $W$ gives the interaction strengths between each pair of neurons and describes how the neurons are connected within the network.
The constants $c_i$ indicate constant inputs from outside the network.

In a \emph{globally stable} dynamical network $(F,X)$, the state of the network tends toward an equilibrium irrespective of its initial condition.
That is, the network has a \emph{globally attracting fixed point} $\mathbf{y} \in X$ such that for any $\mathbf{x} \in X$, $F^k(\mathbf{x}) \to \mathbf{y}$ as $k \to \infty$.

Global stability is observed in a number of important systems including neural networks \cite{neural1,neural2,neural3,neural4,neural5}, epidemic models \cite{epidemic}, and the study of congestion in computer networks \cite{ComputerNetworks}.
In such systems the globally attracting equilibrium is typically a state in which the network can carry out a specific task.
Whether or not this equilibrium stays stable depends on a number of factors including external influences but also internal processes such as the network's own growth, both of which can destabilize the network.

To give a sufficient condition under which a network $(F,X)$ is stable, we define a \emph{Lipschitz matrix} (called a \emph{stability matrix} in \cite{BunWebb1}).

\begin{definition}{\textbf{\emph{(Lipschitz Matrix)}}}\label{Lipschitz Matrix}
For $F:X \to X$ suppose there are finite constants $a_{ij} \ge 0$ such that
\[
d_i(F_i(\mathbf{x}),F_i(\mathbf{y})) \le \sum_{j=1}^{n} a_{ij} d_j(x_j,y_j) \quad \text{for all } \mathbf{x},\mathbf{y} \in X.
\]
Then we call $A=[a_{ij}] \in \mathbb{R}^{n\times n}$ a \emph{Lipschitz matrix} of the dynamical network $(F,X)$.
\end{definition}

It is worth noting that if $A$ is a Lipschitz matrix of a dynamical network then any matrix $B\preceq A$, where $\preceq$ denotes the element-wise inequality, is also a Lipschitz matrix of the network.
However, if the function $F:X\to X$ is piecewise differentiable and each $X_i \subseteq \mathbb{R}$ then the matrix $A \in \mathbb{R}^{nxn}$ given by
\begin{equation}\label{StabilityMatrix}
a_{ij} = \sup_{x \in X} \left\lvert\frac{\partial F_i}{\partial x_j}(\mathbf{x})\right\rvert
\end{equation}
is the Lipschitz matrix of minimal spectral radius of $(F,X)$ (see \cite{BunWebb1}).
From a computational point of view, the Lipschitz matrix $A=[a_{ij}]$ of $(F,X)$ can be more straightforward to find by use of Equation (\ref{StabilityMatrix}) if the function $F:X\to X$ is differentiable, compared to the more general formulation in Definition \ref{Lipschitz Matrix}.

Using Equation (\ref{StabilityMatrix}) it follows that the Lipschitz matrix $A$ of the Cohen-Grossberg neural network from Example \ref{DynamicalNetworkExample} is given by
\begin{equation}\label{CGN_A}
a_{ij} =
\begin{cases}
\text{ } \left\lvert1-\epsilon\right\rvert+K \left\lvert W_{ii}\right\rvert & \text{if } i=j \\
\text{ } K\left\lvert W_{ij}\right\rvert & \text{otherwise.}
\end{cases}
\end{equation}

It is straightforward to verify that a Lipschitz matrix exists for a dynamical network $(F,X)$ if and only if the mapping $F$ is Lipschitz continuous.
The idea is to use the Lipschitz matrix to simplify the stability analysis of nonlinear networks, using the following theorem of \cite{BunWebb1}.
Here
\[\rho(A)=\max_{\lambda \in \sigma(A)}|\lambda|\]
denotes the spectral radius of a matrix $A$, where $\sigma(A)$ are the eigenvalues of $A$.

\begin{theorem}{\textbf{\emph{(Network Stability)}}}\label{NetworkStability}
Let $A$ be a Lipschitz matrix of a dynamical network $(F,X)$.
If $\rho(A)<1$, then $(F,X)$ is stable.
\end{theorem}

It is worth noting that if we use the Lipschitz matrix $A$ to define the dynamical network $(G,X)$ by $G(\mathbf{x})=A\mathbf{x}$ then $(G,X)$ is stable if and only if $\rho(A)<1$.
Thus, a Lipschitz matrix of a dynamical network $(F,X)$ can be thought of as the worst-case linear approximation to $F$.
If this approximation has a globally attracting fixed point, then the original dynamical network $(F,X)$ must also be stable.
Note, however, that the condition that $\rho(A)<1$ is sufficient but not necessary for $(F,X)$ to be stable.
In fact, this stronger condition implies much more than network stability, so following the convention introduced in \cite{BunWebb1} we assign it the name of \emph{intrinsic stability}.

\begin{definition}{\textbf{\emph{(Intrinsic Stability)}}}
Let $A \in \mathbb{R}^{nxn}$ be a Lipschitz matrix of a dynamical network $(F,X)$.
If $\rho(A)<1$, then we say $(F,X)$ is \emph{intrinsically stable}.
\end{definition}

The Cohen-Grossberg neural network $(C,\mathbb{R}^{n\times n})$ has the stability matrix $A=\left\lvert1-\epsilon\right\rvert I+K\left\lvert W \right\rvert$ given by Equation (\ref{CGN_A}) with spectral radius
\[
\rho(A) = \left\lvert1-\epsilon \right\rvert+K\rho(\left\lvert W\right\rvert)
\]
Here, $\left\lvert W\right\rvert\in\mathbb{R}^{n\times n}$ is the matrix $W$ in which we take the absolute value of each entry.
Thus, $(C,\mathbb{R}^{n\times n})$ is intrinsically stable if $\left\lvert1-\epsilon\right\rvert+K\rho(\left\lvert W\right\rvert)<1$.

\section{Constant-Time-Delayed Dynamical Networks}\label{sec:3}
As mentioned in the introduction, the dynamics of most real networks are \emph{time-delayed}.
That is, an interaction between two network elements will typically not happen instantaneously but will be delayed due to either the physical separation of these elements, their finite processing speeds, or be delayed due to other factors.
We formalize this phenomenon by introducing a \emph{delay distribution matrix} $D=[d_{ij}]$ into a dynamical network $(F,X)$, where each $d_{ij}$ is a nonnegative integer denoting the constant number of discrete time-steps by which the interaction from the $j^{th}$ network element to the $i^{th}$ network element is delayed.

\begin{definition}{\textbf{\emph{(Constant Time-Delayed Dynamical Network)}}} \label{ConstantDelayNetwork}

\noindent Let $(F,X)$ be a dynamical network and $D=[d_{ij}]\in\mathbb{N}^{n\times n}$ a delay distribution matrix with $\max_{i,j}d_{ij}\le L$, a bound on the delay length. Let $X_L$, the \emph{extension of} $X$ to delay-space, be defined as
\[
X_L = \bigoplus_{\ell=0}^L \bigoplus_{i=1}^n X_{i,\ell}
\quad \text{ where } \quad X_{i,\ell}=X_i \quad \text{ for } \quad 1 \le i \le n \quad \text{ and } \quad 0 \le \ell \le L.
\]
Componentwise, define $F_D:X_L \to X_L$ by
\begin{equation}\label{identity}
(F_D)_{i,\ell+1}:X_{i,\ell} \to X_{i,\ell+1} \quad \text{ given by the identity map } \quad (F_D)_{i,\ell+1}(x_{i,\ell}) = x_{i,\ell}\quad \text{ for } \quad 0 \le \ell \le L-1
\end{equation}
and
\begin{equation}\label{same-but-later}
(F_D)_{i,0}:\bigoplus_{j=1}^n X_{j,d_{ij}} \to X_{i,0} \quad \text{ given by } \quad (F_D)_{i,0} = F_i(x_{1,d_{i1}},x_{2,d_{i2}},\hdots,x_{n,d_{in}})
\end{equation}
where $F_i:X\to X_i$ is the $i^{th}$ component function of $F$ for $i=1,2,\hdots,n$.
Then $(F_D,X_L)$ is the \emph{delayed version} of $F$ corresponding to the fixed-delay distribution $D$ with \emph{delay bound} $L$.
\end{definition}

We order the component spaces of $X_L$ in the following way.
If $\mathbf{x}\in X_L$ then
\[
\mathbf{x}=[x_{1,0},x_{2,0},\hdots,x_{n,0},x_{1,1},x_{2,1},\hdots,x_{n,L}]^T
\]
where $x_{i,\ell}\in X_{i,\ell}$ for $i=1,2,\hdots,n$ and $\ell=0,1,\hdots,L$.

The formalization in Definition \ref{ConstantDelayNetwork} captures the idea of adding a time-delay of length $d_{ij}\leq L$ to an interaction:
Each $X_i$ is effectively copied $L$ times, and past states of the $i^{th}$ element are passed down this chain by the identity component functions $(F_D)_{i,\ell+1}$ for $0 \le \ell \le L-1$ in Equation (\ref{identity}). When a state of the $i^{th}$ element has been passed through the chain $d_{ij}$ times over $d_{ij}$ time-steps it then influences the $i^{th}$ network element, as described by the entry-wise substitutions of $x_{j,d_{ij}}$ for $x_j$ in $(F_D)_{i,0}$ in Equation (\ref{same-but-later}).

\begin{figure}
\begin{center}
    \begin{overpic}[scale=.33]{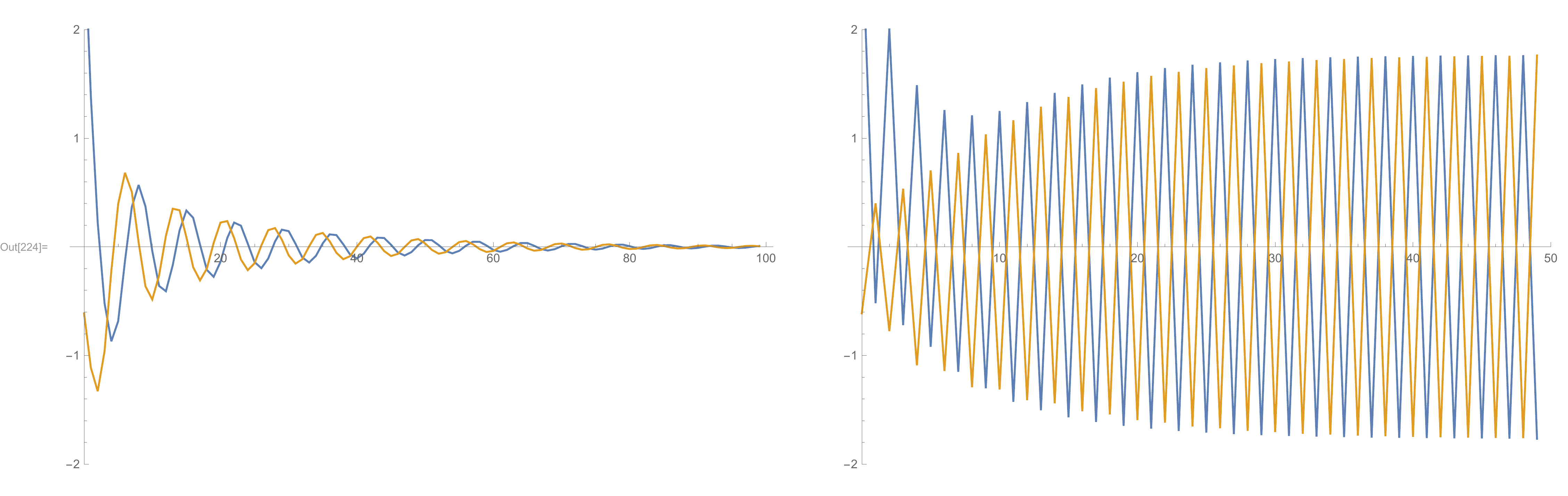}
    \put(11,-1){Stable Network $(C,X)$}
    \put(56,-1){Constant Time-Delayed Network $(C_D,X_3)$}
    \end{overpic}
\end{center}
  \title{}
  \caption{Left: The stable dynamics of the two-neuron Cohen-Grossberg network $(C,X)$ from Example \ref{ConstantDelayExample} is shown. Right: The unstable dynamics of the constant time-delayed version of this network $(C_D,X_3)$ is shown with the delay distribution given by the matrix $D$ in Equation \eqref{Dmatrix}.}\label{fig:1}
\end{figure}

\begin{example}\label{ConstantDelayExample}
Consider a simple 2-neuron version of a Cohen-Grossberg neural network $(C,X)$ given by
\begin{equation}\label{short_delay_CGN}
C \bigg(\begin{bmatrix}
x_1 \\
x_2
\end{bmatrix} \bigg) =
\begin{bmatrix}
    C_1(x_1,x_2) \\
    C_2(x_1,x_2)
\end{bmatrix} =
\begin{bmatrix}
    (1-\epsilon)x_1+W_{11}\phi(x_1)+W_{12}\phi(x_2)+c_1 \\
    (1-\epsilon)x_2+W_{21}\phi(x_1)+W_{22}\phi(x_2)+c_2
\end{bmatrix},
\end{equation}
in which $X = \mathbb{R}^2$. For the delay distribution $D$ given by
\begin{equation}\label{Dmatrix}
D=\begin{bmatrix}
1&2\\
1&3
\end{bmatrix},
\end{equation}
which has a maximum delay of $L=3$ the time-delayed network $(C_D,X_3)$ is given by
\[
C_D \left(\begin{bmatrix}
x_{1,0}\\
x_{2,0}\\
x_{1,1}\\
x_{2,1}\\
x_{1,2}\\
x_{2,2}\\
x_{1,3}\\
x_{2,3}
\end{bmatrix} \right) =
\begin{bmatrix}
    C_1(x_{1,1},x_{2,2}) \\
    C_2(x_{1,1},x_{2,3}) \\
    x_{1,1} \\
    x_{2,1} \\
    x_{1,2} \\
    x_{2,2} \\
    x_{1,3} \\
    x_{2,3}
\end{bmatrix} =
\begin{bmatrix}
    (1-\epsilon)x_{1,1}+W_{11}\phi(x_{1,1})+W_{12}\phi(x_{2,2})+c_1 \\
    (1-\epsilon)x_{2,3}+W_{21}\phi(x_{1,1})+W_{22}\phi(x_{2,3})+c_2 \\
    x_{1,0} \\
    x_{2,0} \\
    x_{1,1} \\
    x_{2,1} \\
    x_{1,2} \\
    x_{2,2}
\end{bmatrix}
\]
in which $X_3=\mathbb{R}^8$.
The time-delayed network $(C_D,X_3)$ is the same as the original network $(C,X)$ except that the state of $x_{1,0}$ gets passed through one identity mapping before it is input into $F_1$ and twice before it is input into $F_2$.
Similarly, $x_{2,0}$ gets passed through one identity mapping before it is input into $F_1$ and three identity mappings before it is input into $F_2$.

A natural question is to ask is whether constant time-delays affect the stability of a network.
We note that if \[
W=\begin{bmatrix}
0 & -\frac{3}{4} \\
\frac{3}{4} & 0
\end{bmatrix},
\]
$c_1=c_2=0$, $\sigma(x)=tanh(x)$, and $\epsilon=\frac{2}{5}$ then the dynamical network $(C,X)$ given in Equation (\ref{short_delay_CGN}) is stable as can be seen in Figure (\ref{fig:1}) (left).
However, the time-delayed version $(C_D,X_3)$ of this network is unstable as is shown in the same figure (right).
That is, the time-delays given by the delay distribution $D$ have a \emph{destabilizing effect} on this network.
\end{example}

An important fact about the network constructed in this example is that its spectral radius
\[
\rho(A)=\left\lvert 1-\epsilon\right\rvert+K\rho(\left\lvert W\right\rvert)=1.35>1.
\]
That is, although $(C,X)$ is stable it is not intrinsically stable.

In \cite{BunWebb1}, the authors demonstrate that intrinsically stable systems are resilient to the addition of constant time-delays, as is stated in the following theorem.

\begin{theorem}{\textbf{\emph{(Intrinsic Stability and Constant Delays)}}}\label{ConstantIntrinsic}
Let $(F,X)$ be a dynamical network and $D=[d_{ij}]$ a delay-distribution matrix.
Let $L$ satisfy $\max_{i,j}d_{ij}\le L$.
Then $(F,X)$ is intrinsically stable if and only if $(F_D,X_L)$ is intrinsically stable.
\end{theorem}

Beyond maintaining stability, we note that any fixed point(s) of an undelayed network $(F,X)$ will also be fixed point(s) of any delayed version $(F_D,X_L)$.
This is formalized in the following proposition, and proven in the Appendix.
Before stating this proposition, we require the following definition.

\begin{definition}{\textbf{\emph{(Extension of a Point to Delay-Space)}}}
Let $E_L(\mathbf{x})\in X_L$ be equal to $L+1$ copies of $\mathbf{x}\in X$ stacked into a single vector, namely
\[
E_L(\mathbf{x}) =
\begin{bmatrix}
\mathbf{x}_0 \\
\mathbf{x}_1 \\
\vdots \\
\mathbf{x}_L
\end{bmatrix} \quad \emph{ where } \quad \mathbf{x}_\ell = \mathbf{x}
\quad \emph{ for } \quad 0 \le \ell \le L.
\]
\end{definition}

\begin{proposition}{\textbf{\emph{(Fixed Points of Delayed Networks)}}}\label{FixedPoint}
Let $\mathbf{x}^*$ be a fixed point of a dynamical network $(F,X)$.
Then for all delay distributions $D$ with $\max_{ij}d_{ij}\le L$, $E_L(\mathbf{x}^*)$ is a fixed point of $(F_D,X_L)$.
\end{proposition}

As an immediate consequence of Proposition \ref{FixedPoint} and Theorem \ref{ConstantIntrinsic}, if an undelayed network $(F,X)$ is intrinsically stable with a globally attracting fixed point $\mathbf{x}^*\in X$, then the delayed version $(F_D,X_L)$ will have the ``same'' globally attracting fixed point $E_L(\mathbf{x}^*)$, in that $\mathbf{x}^*$ is the restriction of $\mathbf{y}=E_L(\mathbf{x}^*)$ to the first $n$ component spaces of $X_L$. Thus, the asymptotic dynamics of an intrinsically stable network and any version of the network with constant time delays are essentially identical.

\begin{figure}
\begin{center}
    \begin{overpic}[scale=.33]{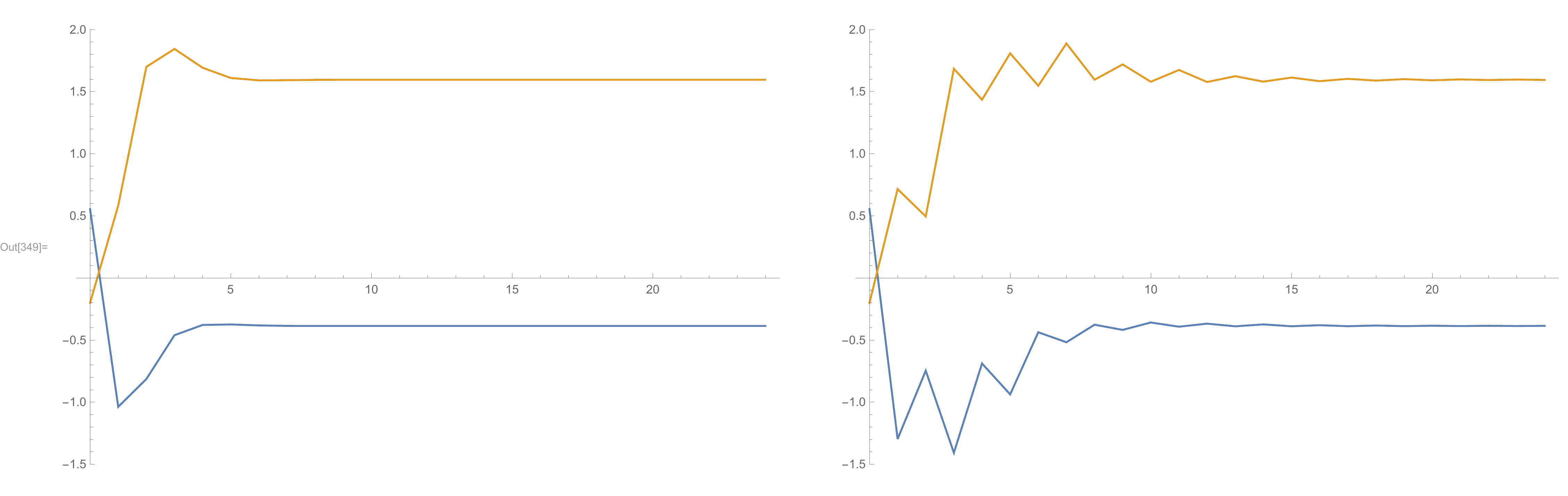}
    \put(8,-1){Intrinsically Stable Network $(C,X)$}
    \put(56,-1){Constant Time-Delayed Network $(C_D,X_3)$}
    \end{overpic}
\end{center}
  \caption{Left: The dynamics of the intrinsically stable network $(C,X)$ from Example \ref{ex:2} is shown. Right: The stable dynamics of the constant time-delayed version of this network $(C_D,X_3)$ is shown with the delay distribution given by the matrix $D$ in Equation \eqref{Dmatrix}. Both systems are attracted to the fixed point $\mathbf{x}^*=(-.386,1.595)$.}\label{fig:2}
\end{figure}

\begin{example}\label{ex:2}
Consider again the Cohen-Grossberg neural network $(C,X)$ and delay matrix $D$ given in Example \ref{ConstantDelayExample} where $W$ and $\sigma$ are as before but $\epsilon=\frac{4}{5}$, $c_1=-1$ and $c_2=1$. Since $\left\lvert1-\epsilon\right\rvert+\rho(\left\lvert W\right\rvert)=.95<1$ then $(C,X)$ is intrinsically stable with the globally attracting fixed point $\mathbf{x}^*=(-.386,1.595)$ as seen in Figure \ref{fig:2} (left).
Since $(C,X)$ is intrinsically stable then Theorem \ref{ConstantIntrinsic} together with Proposition \ref{FixedPoint} imply that not only is $(C_D,X_3)$ stable but its globally attracting fixed point is $E_3(\mathbf{x}^*)$.
This is shown in Figure \ref{fig:2} (right).
\end{example}

These results justify the modeling choice of ignoring constant time-delays when analyzing intrinsically stable real-world networks. However, these results do not account for the potential of time-dependent delays arising from external or stochastic influences, etc. The main results of this paper, presented in the next section, focus on strengthening the conclusion of Theorem \ref{ConstantIntrinsic}.

\section{Time-Varying Time-Delayed Dynamical Networks}\label{sec:4}
As the title of this section suggests, constant time-delays are not the only type of time delays that can occur in dynamical networks. More importantly, time-delays that vary with time occur in many real-world networks and in such systems are a significant source of instability \cite{IntroPaper, destabilizing1, destabilizing2}. It is worth noting that modeling such delays introduces even more complexity into models of dynamical networks that can already be quite complicated. This can hinder the tractability of analyzing such systems.

In order to define a network with time-varying time-delays, we first define the more general concept of a \emph{switched network}.

\begin{definition}{\textbf{\emph{(Switched Network)}}}
Let $M$ be a set of Lipschitz continuous mappings on $X$, such that for every $F\in M$, $(F,X)$ is a dynamical network.
Then we call $(M,X)$ a \emph{switched network} on $X$.
Given some sequence $\{F^{(k)}\}_{k=1}^\infty \subset M$, we say that $(\{F^{(k)}\}_{k=1}^\infty,X)$ is an instance of $(M,X)$, with orbits determined at time $k$ by the function
\[
\mathscr{F}^{k}(\mathbf{x}) = F^{(k)} \circ \hdots \circ F^{(2)} \circ F^{(1)}(\mathbf{x}) \quad \text{ for } \quad \mathbf{x}\in X.
\]
For the switched network $(M,X)$ we construct a \emph{Lipschitz set} $S$ consisting of a set of $n\times n$ matrices as follows:
For each $F\in M$, contribute exactly one Lipschitz matrix $A$ of $(F,X)$ to the set $S$.
\end{definition}

$(M,X)$ is an ensemble of dynamical systems formed by taking all possible sequences of mappings $\{F^{(k)}\}_{k=1}^\infty \subset M$.
For a switched network $(M,X)$, the set $S$ serves an analogous purpose to the Lipschitz matrix $A$ of a dynamical network $(F,X)$, as we soon demonstrate.
Before describing this we consider the following example.

\begin{example}\label{CounterSwitching}
Let $(P,\mathbb{R}^2)$ and $(Q,\mathbb{R}^2)$ be the simple dynamical networks given by
\[
P(\mathbf{x}) = \begin{bmatrix}
\epsilon & 1 \\
0 & \epsilon
\end{bmatrix} \mathbf{x}
\quad \text{ and } \quad
P(\mathbf{x}) = \begin{bmatrix}
\epsilon & 0 \\
1 & \epsilon
\end{bmatrix} \mathbf{x} \quad \text{ for small } \quad \epsilon<<1.
\]
For $M=\{P,Q\}$ let $\{F^{(k)}\}_{k=1}^\infty$ be the sequence that alternates between $P$ and $Q$, i.e. $F^{(k)}=P$ if $k$ is odd and $F^{(k)}=Q$ if $k$ is even.
Note that if we let $U=Q\circ P$ then
\[
U(\mathbf{x}) = \begin{bmatrix}
\epsilon^2 & \epsilon \\
\epsilon & 1+\epsilon^2
\end{bmatrix}\quad \text{ with } \quad
\rho(U)=\frac{1}{2}(1+2\epsilon^2+\sqrt{1+4\epsilon^2})>1.
\]
Since $\mathcal{F}^{(2k)}(\mathbf{x})=U\circ\hdots\circ U(\mathbf{x})$ for $\mathbf{x}\in\mathbb{R}^2$ then, as $U(\mathbf{x})$ is a linear system $\lim_{k\to\infty}\mathcal{F}^{(2k)}(\mathbf{x})=\infty$ for any $\mathbf{x}\ne0$.
This is despite the fact that both $(P,\mathbb{R}^2)$ and $(Q,\mathbb{R}^2)$ are intrinsically stable both having the globally attracting fixed point $\mathbf{0}$.
\end{example}

The issue is that although the spectral radius of both $P$ and $Q$ in this example are arbitrarily small, their joint spectral radius is not.

\begin{definition}{\textbf{\emph{(Joint Spectral Radius)}}}\label{JointSpectraldefinition}
Given some $\mathbf{z}^0\in \mathbb{R}^n$ and some set of matrices $S\subset \mathbb{R}^{n\times n}$, let $\mathbf{z}^k = A_k\hdots A_2A_1\mathbf{z}^0$ for some sequence $\{A_i\}_{i=1}^\infty\subset S$.
The \emph{joint spectral radius} $\overline{\rho}(S)$ of the set of matrices $S$ is the smallest value $\overline{\rho}\ge 0$ such that for every $\mathbf{z}^0\in \mathbb{R}^n$ there is some constant $C>0$ for which
\[
||{\mathbf{z}^k}|| \le C (\overline{\rho})^k.
\]
\end{definition}

It is known that $\{\mathbf{z}^k\}_{k=1}^\infty$ converges to the origin for all $\mathbf{z}^0\in \mathbb{R}^n$ if and only if $\overline{\rho}(S) < 1$ \cite{JointSpectral}.
This allows us to state the following result regarding the asymptotic behavior of a nonlinear switched network $(M,X)$ whose Lipschitz set $S$ satisfies $\overline{\rho}(S) < 1$.

\begin{result}{\textbf{\emph{(Independence of Initial Conditions for Switched Networks)}}}{\label{LimitOrbits}}
Let $S$ be a Lipschitz set of a switched network $(M,X)$ satisfying $\overline{\rho}(S) < 1$, and let $(\{F^{(k)}\}_{k=1}^\infty,X)$ be an instance of $(M,X)$.
Then for all initial conditions $\mathbf{x}^0,\mathbf{y}^0\in X$, there exists some $C>0$ such that
\[
d_{max}(\mathscr{F}^k(\mathbf{x}^0),\mathscr{F}^k(\mathbf{y}^0)) \le C\overline{\rho}(S)^k.
\]
Additionally, if $\mathbf{x}^*$ is a shared fixed point of $(F,X)$ for all $F\in M$, then $\lim_{k\to\infty}\mathscr{F}^k(\mathbf{x}^0)=\mathbf{x}^*$ for all initial conditions $\mathbf{x}^0\in X$.
\end{result}

Hence, if the joint spectral radius of the Lipschitz set $S$ of $M$ is less than $1$, all orbits in a switched network become asymptotically close to one another as time goes to infinity.
Even if this limit-orbit is not convergent to any fixed point, this result implies an asymptotic independence to initial conditions. An example of this is the following.

\begin{example}\label{ex:switchednet}
Let $(G,X)$ and $(H,X)$ be the Cohen-Grossberg neural networks given by
\begin{equation*}
G \bigg(\begin{bmatrix}
x_1 \\
x_2
\end{bmatrix} \bigg) =
\begin{bmatrix}
    (1-\epsilon_1)x_1-\frac{3}{4}\phi(x_2)+c_1 \\
    (1-\epsilon_1)x_2+\frac{3}{4}\phi(x_1)+c_2
\end{bmatrix} \ \ \text{and} \ \
H \bigg(\begin{bmatrix}
x_1 \\
x_2
\end{bmatrix} \bigg) =
\begin{bmatrix}
    (1-\epsilon_2)x_1+\frac{1}{4}\phi(x_2)+d_1 \\
    (1-\epsilon_2)x_2+\frac{1}{4
    }\phi(x_1)+d_2
\end{bmatrix},
\end{equation*}
respectively, in which $\sigma(x)=\tanh(x)$, $\epsilon_1=\frac{4}{5}$, $\epsilon_2=\frac{3}{10}$, $c_1=d_2=-1$, and $c_2=d_1=1$. Setting $M=\{G,H\}$ then $S$ is the Lipschitz set of $M$ given by

\[
S=\left\{
\begin{bmatrix}
\frac{1}{5}&\frac{3}{4}\\
\frac{3}{4}&\frac{1}{5}
\end{bmatrix},
\begin{bmatrix}
\frac{7}{10}&\frac{1}{4}\\
\frac{1}{4}&\frac{7}{10}
\end{bmatrix}
\right\}.
\]
It can be shown that the joint spectral radius $\overline{\rho}(S)=.95$ (by use of Proposition \ref{RIradius} in Section \ref{sec:5}). As this is less than 1, then for any instance $(\{F^{(k)}\}_{k=1}^\infty,X)$ of $M$ and any initial conditions $\mathbf{x}^0$ and $\mathbf{y}^0$, we have
\[
\lim_{k\rightarrow\infty}d_{max}(\mathcal{F}^k(\mathbf{x}^0),\mathcal{F}^k(\mathbf{y}^0))=0
\]
by the Main Result \ref{LimitOrbits}. This can be seen in Figure \ref{fig:switched} (left) where $(\{F^{(k)}\}_{k=1}^\infty,X)$ is the instance given by $\{F^{(k})\}_{k=1}^\infty=\{G,G,G,H,H,H,\dots\}$.

If we set $c_1=c_2=d_1=d_2=0$ in both $(G,X)$ and $(H,X)$ then both systems have the shared fixed point $\mathbf{0}$. In this case Main Result \ref{LimitOrbits} indicated that any instance $(\{F^{(k)}\}_{k=1}^\infty,X)$ of $M$ will be stable with the globally attracting fixed point $\mathbf{0}$. This is shown in Figure \ref{fig:switched} (right) where again $\{F^{(k})\}_{k=1}^\infty=\{G,G,G,H,H,H,\dots\}$.
\end{example}

\begin{figure}
\begin{center}
    \begin{overpic}[scale=.23
    ]{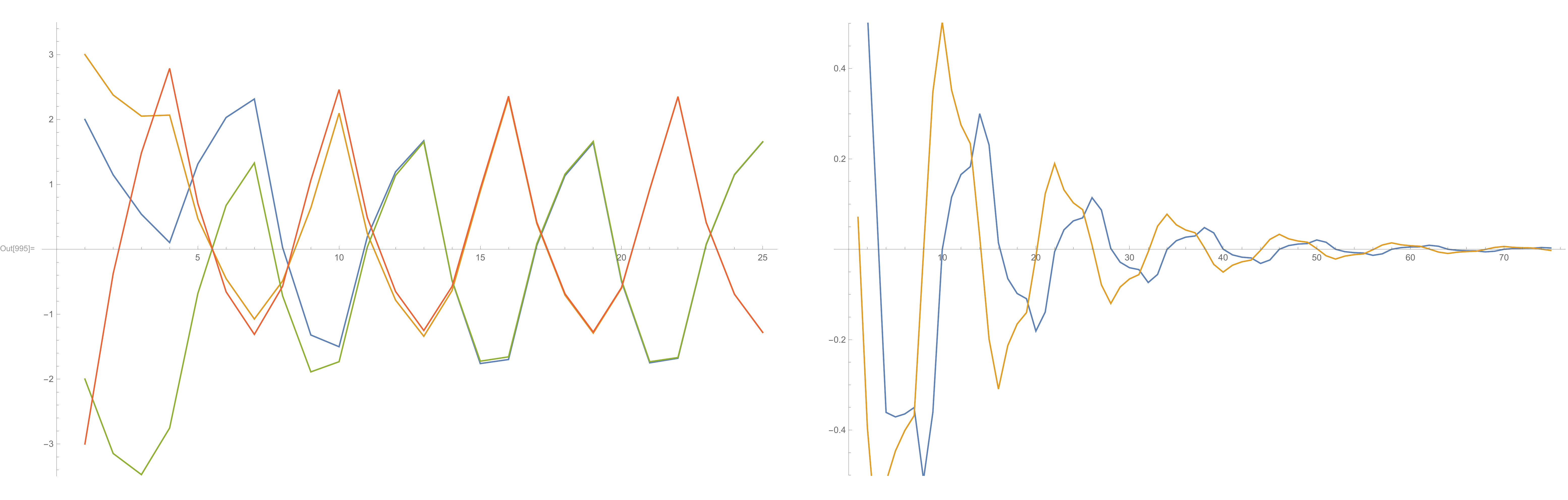}
    \put(2,-2){Independence to Initial Conditions} 
    \put(53,-2){Dynamics with a Shared Fixed Point} 
    \end{overpic}
\end{center}
  \caption{Left: The dynamics of the switched network in Example \ref{ex:switchednet} is shown for two different initial conditions $\mathbf{x}^0=(2,3)$ (shown in blue and yellow) and $\mathbf{y}^0=(-2,-3)$ (shown in green and red). As the corresponding joint spectral radius $\overline{\rho}(S)$ of the network is less than 1, the orbits of these initial conditions converge to each other. Right: Modifying this switched network so that both $F,G\in M$ have the shared fixed point $\mathbf{0}$ results in a stable switched system with the globally attracting fixed point $\mathbf{0}$.}\label{fig:switched}
\end{figure}

Thus, as might be expected from the complicated nature of a switched system, the condition $\overline{\rho}(S)<1$ alone is not able to match the strong implication of global stability, as $\rho(A)<1$ does for a dynamical network (see Theorem [\ref{ConstantIntrinsic}]).
Furthermore, it is often notoriously difficult to compute or approximate the joint spectral radius $\overline{\rho}(S)$ of a general set of matrices $S$ \cite{JointSpectral, JSRHard}.

Somewhat surprisingly, these issues are resolved when our switched system arises from a network experiencing time-varying time-delays.
In this case, the computation of the joint spectral radius reduces to computing the spectral radius of the Lipschitz matrix of the original undelayed dynamical network $(F,X)$, which for even large systems can be done efficiently using the power method \cite{powermethod}. This provides a general and computationally efficient method for verifying asymptotic stability despite time-varying time-delays.

\begin{result}{\textbf{\emph{(Intrinsic Stability and Time-Varying Time-Delayed Networks)}}}\label{TimeVaryingIntrinsic}
Suppose $(F,X)$ is intrinsically stable with $\rho(A)<1$, where $A\in\mathbb{R}^{n\times n}$ is a Lipschitz matrix of $F$ and $\mathbf{x}^\ast$ is the network's globally attracting fixed point.
Let $L>0$ and
\[
M_d=\{F_D|D\in\mathbb{N}^{n\times n}\text{ with }\max_{ij}d_{ij}\le L\}
\]
and let $S_d$ be the Lipschitz set of $M_d$.

\noindent Then $E_L(\mathbf{x}^*)$ is a globally attracting fixed point of every instance $(\{F_{D^{(k)}}\}_{k=1}^\infty,X_L)$ of $(M_d,X_L)$.
Furthermore, $\overline{\rho}(S_d)=\rho(A_L)<1$, where
\[
A_L = \begin{bmatrix}
\mathbf{0}_{n\times nL} & A \\
\mathbf{I}_{nL\times nL} & \mathbf{0}_{nL\times n}
\end{bmatrix}.
\]
\end{result}

Hence, any intrinsically stable dynamical network $(F,X)$ retains convergence to the same equilibrium even when it experiences time-varying time-delays.
This extends Theorem 2.3 of \cite{BunWebb1} to the much larger and more complicated class of switching-delay networks. Furthermore, note that intrinsic stability is a delay-independent result, which makes no assumption regarding the rate of growth of the time delays. In Main Result \ref{TimeVaryingIntrinsic} we call $\overline{\rho}(S_d)$ the \emph{convergence rate} of the system, since it provides the exponential bound on the rate at which all orbits converge to the fixed point $E_L(\mathbf{x}^*)$.

It is worth emphasizing that as a consequence of this result, to determine the asymptotic behavior of a switched network $(M,X_L)$ with $M=\{F_D|D\in\mathbb{N}^{n\times n}\text{ with }\max_{ij}d_{ij}\le L\}$, in which the presence and magnitude of time delays is not exactly known, it suffices to study the dynamics of the much simpler undelayed system $(F,X)$.

\begin{figure}
\begin{center}
    \begin{overpic}[scale=.4]{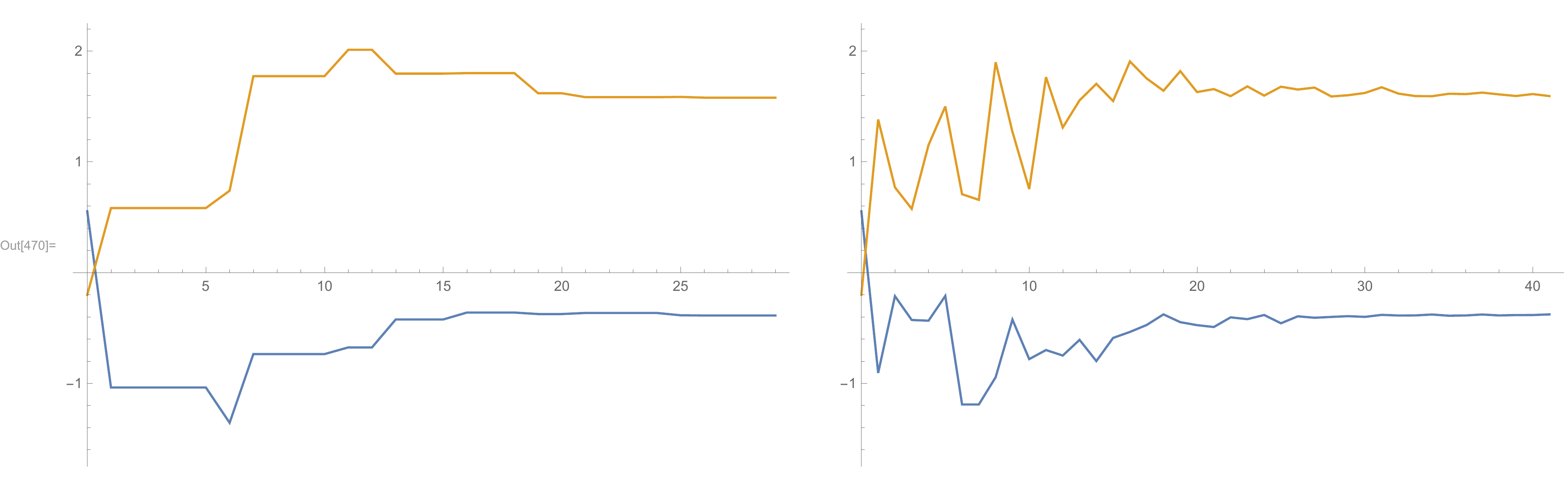}
    \put(5,-1){Periodically Delayed Network $(C_P,X_5)$}
    \put(56,-1){Stochastically Delayed Network $(C_U,X_{10})$}
    \end{overpic}
\end{center}
  \caption{Left: The dynamics of the intrinsically stable two-neuron Cohen-Grossberg network $(C,X)$ from Example \ref{ex:2} is shown in which the network has periodic time-varying time-delays. Right: The dynamics of the same Cohen-Grossberg network is shown in which the network has stochastic time-varying time-delays. Both systems are attracted to the fixed point $\mathbf{x}^*=(-.386,1.595)$ similar to the behavior shown in Figure \ref{fig:2}.}\label{fig:3}
\end{figure}

\begin{example}{\textbf{\emph{(Periodic and Stochastic Time-Varying Delays)}}}\label{PeriodicStochasticExample}
Consider the intrinsically stable Cohen-Grossberg neural network $(C,X)$ from Example \ref{ex:2} with the exception that we let $\{D^{(k)}\}_{k=1}^\infty$ be the sequence of delay distributions given by
\begin{equation}\label{periodicD}
D^{(k)} = \begin{bmatrix}
k\mod 5 &\quad k\mod 6 \\
k\mod 6 &\quad k\mod 5
\end{bmatrix}.
\end{equation}
These time-varying delays are periodic with period $30$ and $L=5$. The result of these delays on the network's dynamics can be seen in Figure \ref{fig:3} (left) where we let $(C_P,X_5)$ denote the network $(C,X)$ with time-varying time-delays given by \eqref{periodicD}.
Note that although the trajectories are altered by these delays they still converge to the same fixed point $\mathbf{x}^*=(-.386,1.595)$ as in the undelayed and constant time-delayed networks (cf. Figure \ref{fig:2}) as guaranteed by Main Result \ref{TimeVaryingIntrinsic}.

If instead we let
\begin{equation}\label{uniformD}
D^{(k)}_{U[0,10]} \in \mathbb{N}^{2\times 2}
\end{equation}
be the random matrix in which each entry is sampled uniformly from the integers $\{0,1,\hdots,10\}$ the resulting switched network is still stable, as guaranteed by Main Result \ref{TimeVaryingIntrinsic}. The network's trajectories still converge to the point $\mathbf{x}^*=(-.386,1.595)$ as shown in Figure \ref{fig:3} (right) where we let $(C_U,X_{10})$ denote the network $(C,X)$ with time-varying time-delays given by \ref{uniformD}.
\end{example}

\subsection{Application: Linear Systems with Distinct Delayed and Undelayed Interactions}\label{IntervalDelays}
A number of papers have published delay-dependant results regarding time-varying time-delayed systems whose delayed interactions are separated from the undelayed interactions (see , for instance, \cite{dependent1,dependent11,dependent12,dependent13,dependent14}). By \emph{delay-dependent} we mean that the criteria that determines whether a network is stable depends on the length of the delays the network experiences, as opposed to the delay-independent results of Section \ref{sec:4}. We demonstrate how to analyze these by use of Main Result \ref{TimeVaryingIntrinsic}, and compare our results to those in the literature.

Consider a linear system of the form
\begin{equation}\label{unlifted}
\mathbf{x}^{k+1} = A\mathbf{x}^k + B\mathbf{x}^{k-\tau(k)} \quad \text{where}\quad 1\le \tau(k) \le L \quad \text{for}\quad k\geq 0
\end{equation}
where $A,B\in\mathbb{R}^{n\times n}$ are constant matrices, and $\tau(k)$ is a positive integer representing the magnitude of the time-varying time-delay, bounded by some $L>0$.
Here $A$ and $B$ represent distinct weights of the delayed and undelayed interactions.

The minimally delayed version of system (\ref{unlifted}) is given by
\begin{equation}\label{UnliftedUndelayed}
\mathbf{x}^{k+1} = A\mathbf{x}^k + B\mathbf{x}^{k-1},
\end{equation}
which we can express in terms of a single transition matrix $\widetilde{A}$ as
\begin{equation}\label{Lifted}
\widetilde{\mathbf{x}}^{k+1} =
\widetilde{A}\widetilde{\mathbf{x}}^k\quad \text{where}\quad
\widetilde{A} = \begin{bmatrix}
A & B \\
I_{n\times n} & \mathbf{0}_{n\times n}
\end{bmatrix}\quad \text{and}\quad
\widetilde{\mathbf{x}}^k=\begin{bmatrix}
\mathbf{x}^{k} \\
\mathbf{x}^{k-1}
\end{bmatrix}.
\end{equation}
As in \cite{automatica_lifted}, we say that system (\ref{Lifted}) is the \emph{lifted representation} of system (\ref{UnliftedUndelayed}).

Now observe that we can represent system (\ref{unlifted}) in the notation of Main Result \ref{TimeVaryingIntrinsic} as the system $(F,\mathbb{R}^{2n})$ where
\[
F(\widetilde{\mathbf{x}}^{k}) = \widetilde{A}\widetilde{\mathbf{x}}^k
\]
where, since $F$ is linear, the Lipschitz matrix of $F$ is given by $\left\vert\widetilde{A}\right\vert\in\mathbb{R}^{2n \times 2n}$ and the zero vector $\mathbf{0}\in\mathbb{R}^n$ is a fixed point. Then system (\ref{unlifted}) is the switched system instance $(\{F_{D^{(k)}}\}_{k=1}^\infty,X_L)$ obtained from the sequence of delay distributions
\[
D^{(k)} = \tau(k)\begin{bmatrix}
\mathbf{0}_{n\times n} & \mathbf{1}_{n\times n} \\
\mathbf{0}_{n\times n} & \mathbf{0}_{n\times n}
\end{bmatrix}\
\]
ensuring that delays only occur to the delayed interactions modeled by $B$, with magnitude $\tau(k)\leq L$.

It follows immediately from Main Result \ref{TimeVaryingIntrinsic} that the system (\ref{unlifted}) is stable for arbitrary large delay bounds $L$ when $(F,\mathbb{R}^{2n})$ is intrinsically stable, that is when $\rho(\widetilde{A})<1$ is satisfied.

\begin{example}{\textbf{\emph{(Intrinsically Stable)}}}\label{IntStable4.1}
Consider system (\ref{unlifted}) with
\[
A = \begin{bmatrix}
0.6 & 0 \\
0.35 & 0.7
\end{bmatrix},\quad  B = \begin{bmatrix}
0.1 & 0 \\
0.2 & 0.1
\end{bmatrix}.
\]

The transition matrix of the lifted representation is
\[
\widetilde{A} = \begin{bmatrix}
0.6 & 0 & 0.1 & 0 \\
0.35 & 0.7 & 0.2 & 0.1 \\
1&0&0&0 \\
0&1&0&0
\end{bmatrix}
\]
which satisfies $\rho(\left\vert\widetilde{A}\right\vert) \approx 0.822 <1$. Then this system is intrinsically stable, so is in fact stable for an arbitrarily large delay bound $L>0$.
In the following table, we compare this delay-independent result with the delay-dependent results of \cite{dependent1,dependent11,dependent12,dependent13,dependent14}:
\begin{center}
\begin{tabular}{ |c|c| }
 \hline
 Method & Max Upper Bound $L$ \\
 \hline
 Theorem 3.1 of \cite{dependent11} & $10$ \\
 \hline
 Theorem 1, Theorem 2 of \cite{dependent12} & $13$ \\
 \hline
 Theorem 3.2 of \cite{dependent13} & $12$ \\
 \hline
 Theorem 1 of \cite{dependent14} & $15$ \\
 \hline
 Theorem 2 of \cite{dependent1} & $10\cdot 10^{21}$ \\
 \hline
 Main Result \ref{TimeVaryingIntrinsic} of this paper & $\infty$ \\
 \hline
\end{tabular}
\end{center}
It is worth noting that the methods of \cite{dependent1,dependent11,dependent12,dependent13,dependent14} employ techniques involving linear matrix inequalities and Lyapunov functionals. These methods are avoided with a straight-forward application of Main Result \ref{TimeVaryingIntrinsic}, which only requires computing the spectral radius of single $4\times 4$ matrix. As elaborated in Section \ref{sec:6}, checking intrinsic stability is extremely computationally efficient. Furthermore, because intrinsic stability is a delay-independent result, it becomes immediately clear that an intrinsically stable network is stable for arbitrarily large delay bounds $L$.
\end{example}

Beyond improving the results described in Example \ref{IntStable4.1}, in a similar manner, we can also improve the results of Example 2 of \cite{dependent4}. In doing so we replicate the delay-independent result of Example 6.1 of \cite{dependent2}, in which the system is found to be stable for arbitrarily large $L$. However, we do so in a much more computationally efficient manner, without having to solve a series of linear matrix inequalities (see \cite{dependent4}).

There are systems of the form (\ref{unlifted}) which are not intrinsically stable, in which case delay-dependent results provide greater insight.

\begin{example}{\textbf{\emph{(Not Intrinsically Stable)}}}\label{next:ex}
Consider system (\ref{unlifted}) with
\[
A = \begin{bmatrix}
0.8 & 0 \\
0.05 & 0.9
\end{bmatrix},\quad  B = \begin{bmatrix}
-0.1 & 0 \\
-0.2 & -0.1
\end{bmatrix}.
\]
The transition matrix of the lifted representation is
\[
\widetilde{A} = \begin{bmatrix}
0.8 & 0 & -0.1 & 0 \\
0.05 & 0.9 & -0.2 & -0.1 \\
1&0&0&0 \\
0&1&0&0
\end{bmatrix}
\]
which satisfies $\rho(\left\vert\widetilde{A}\right\vert) =1$.
Hence this system is not intrinsically stable, even though \cite{dependent1} showed that it is stable for all $0\le L \le 9.61\times 10^8$.
\end{example}

Even though some systems that are not intrinsically stable turn out to be stable, at least for certain types of delays, the fact that intrinsic stability can be verified with relatively little effort may be reason enough to check. In fact, it may be the case that a spectral radius slightly above or equal to 1 is an indication that a system's stability is resilient to time delays as in Example \ref{next:ex}. However, this is still an open question. For further analysis of the computational complexity of checking intrinsic stability, see section \ref{sec:6}.

\section{Row-Independence Closure of Switched Networks}\label{sec:5}
Using Main Result \ref{LimitOrbits} and Main Result \ref{TimeVaryingIntrinsic}, we can extend our analysis of systems with time-varying time-delays to a more general class of switched networks.

\begin{definition}{\textbf{\emph{(Row-Independence Closure)}}}\label{RowIndependent}
Let $(M,X)$ be a switched network with Lipschitz set $S$.
Then we denote $RI(S)$, the \emph{row-independence closure} of $S$, by
\[
RI(S) = \{A^*|\text{ }\mathbf{a}^*_i = \mathbf{a}^{(i)}_i \text{ for some } A^{(1)},\hdots,A^{(n)}\in S\}
\]
where $\mathbf{a}^{(i)}_i$ denotes the $i^{th}$ row of the $i^{th}$ matrix $A^{(i)}$.
\end{definition}

The idea is that row-independence indicates that there is no conditional relationship between rows of the matrices in $RI(S)$.
Note that $S\subset RI(S)$.

The row-independence closure provides a computationally efficient sufficient condition for satisfying the hypothesis of Main Result \ref{LimitOrbits}.
The following proposition follows directly from Definition \ref{RowIndependent} and the results of \cite{JointSpectral} (restated as Theorem \ref{JSRindpependent} in the Appendix).

\begin{proposition}{\textbf{\emph{(Row-Independence Closure, Intrinsic Stability, and Switched Networks)}}}\label{RIradius}
Let $(M,X)$ be a switched network with Lipschitz set $S$.
Then
\[\overline{\rho}(S)\le \overline{\rho}(RI(S)) = \max_{A\in RI(S)}\rho(A).
\]
\end{proposition}

This allows for the following extension of Main Result \ref{TimeVaryingIntrinsic}, which provides a sufficient condition ensuring that when time-varying time-delays are applied to a stable already-switched system, the new resulting switched system retains stability.

\begin{result}{\textbf{\emph{(Intrinsic Stability and Row-Independent Switched Networks)}}}\label{RIIntrinsic}
Let $(M,X)$ be a switched network with Lipschitz set $S$.
Assume $\mathbf{x}^*$ is a shared fixed point of $(F,X)$ for all $F\in M$ and $\rho(A)<1$ for all $A\in RI(S)$.
Let $L>0$ and
\[
M_d=\{F_D|F\in M,D\in\mathbb{N}^{n\times n}\text{ with }\max_{ij}d_{ij}\le L\}
\]
and let $S_d$ be the Lipschitz set of $M_d$.

\noindent Then $E_L(\mathbf{x}^*)$ is a globally attracting fixed point of every instance $(\{F^{(k)}_{D^{(k)}}\}_{k=1}^\infty,X_L)$ of $(M_d,X_L)$.
Furthermore,
\[\overline{\rho}(S_d)\le \max_{A\in RI(S)}\rho(A_L)<1\]
where given some $A\in\mathbb{R}^{n\times n}$, $A_L$ is defined as
\[
A_L = \begin{bmatrix}
\mathbf{0}_{n\times nL} & A \\
\mathbf{I}_{nL\times nL} & \mathbf{0}_{n\times n}
\end{bmatrix}.
\]
\end{result}

Thus, a switched network with a shared fixed point which also satisfies $\rho(A)<1$ for all $A\in RI(S)$ retains convergence to the same equilibrium, even when it experiences time-varying time-delays. This extends Main Result \ref{TimeVaryingIntrinsic} further to the even more complicated class of switched networks with time-varying time-delays.

\subsection{Application: Switched Linear Systems with Distinct Delayed and Undelayed Interactions}
We now extend our analysis of Section \ref{IntervalDelays} to the case where system (\ref{unlifted}) is also a switched system:
\begin{equation}\label{unliftedSwitched}
\mathbf{x}^{k+1} = A_{\sigma(k)}\mathbf{x}^k + B_{\sigma(k)}\mathbf{x}^{k-\tau(k)} \quad \text{where}\quad 1\le \tau(k) \le L
\end{equation}
where as before, each $A_{\sigma(k)},B_{\sigma(k)}\in\mathbb{R}^{n\times n}$ are constant matrices indexed by $\sigma(k)$, and $\tau(k)$ is a positive integer representing the magnitude of the time-varying time-delay, bounded by some $L>0$. The difference between system (\ref{unliftedSwitched}) and system (\ref{unlifted}) is that in system (\ref{unliftedSwitched}) there are multiple possibilities for the transition weights of the delayed and undelayed interactions given by $A_{\sigma(k)}$ and $B_{\sigma(k)}$, respectively.

The minimally delayed version of system (\ref{unliftedSwitched}) is given by
\begin{equation}\label{UnliftedUndelayedSwitched}
\mathbf{x}^{k+1} = A_{\sigma(k)}\mathbf{x}^k + B_{\sigma(k)}\mathbf{x}^{k-1}
\end{equation}
so the lifted version of system (\ref{UnliftedUndelayedSwitched}) is
\begin{equation}\label{LiftedSwitched}
\widetilde{\mathbf{x}}^{k+1} =
\widetilde{A}_{\sigma(k)}\widetilde{\mathbf{x}}^k\quad \text{where}\quad
\widetilde{A}_{\sigma(k)} = \begin{bmatrix}
A_{\sigma(k)} & B_{\sigma(k)} \\
I_{n\times n} & \mathbf{0}_{n\times n}
\end{bmatrix}\quad \text{and}\quad
\widetilde{\mathbf{x}}^k=\begin{bmatrix}
\mathbf{x}^{k} \\
\mathbf{x}^{k-1}
\end{bmatrix}.
\end{equation}

Now observe that we may represent system (\ref{unliftedSwitched}) in the notation of Main Result \ref{RIIntrinsic} as the system $(M_0,\mathbb{R}^{2n})$ where
\[
M_0 = \{F|\text{ } F(\widetilde{\mathbf{x}}^{k}) = \widetilde{A}_{\sigma(k)}\widetilde{\mathbf{x}}^k\}
\]
where, since each $F$ is linear, the Lipschitz matrix of $F$ is given by $\left\vert\widetilde{A}_{\sigma(k)}\right\vert$ and the zero vector $\mathbf{0}\in\mathbb{R}^n$ is a shared fixed point.
Then system (\ref{unliftedSwitched}) is the switched system instance $(\{F^{(k)}_{D^{(k)}}\}_{k=1}^\infty,X_L)$ obtained from the sequence of transition matrices $\widetilde{A}_{\sigma(k)}$ and the sequence of delay distributions
\[
D^{(k)} = \tau(k)\begin{bmatrix}
\mathbf{0}_{n\times n} & \mathbf{1}_{n\times n} \\
\mathbf{0}_{n\times n} & \mathbf{0}_{n\times n}
\end{bmatrix}\
\]
once the sequences $\sigma(k)$ and $\tau(k)$ are determined. This ensures that delays only occur to the delayed interactions modeled by $B_{\sigma(k)}$, with magnitude $\tau(k)$. It follows immediately from Main Result \ref{RIIntrinsic} that the system (\ref{unliftedSwitched}) is stable for arbitrary large delay bounds $L$ when $(M_0,\mathbb{R}^{2n})$ satisfies the relatively simple condition $\rho(A)<1$ for all $A\in RI(S)$ of the Lipschitz set $S$ of $M_0$.

We apply our analysis to two examples from \cite{automatica_switched} to demonstrate the effectiveness of Main Result \ref{RIIntrinsic}.

\begin{example}{\textbf{\emph{(Row-Independent Switched Network)}}}\label{ex:5.1}
Consider system (\ref{unliftedSwitched}) with
\[
A_1 =A_2= \begin{bmatrix}
0 & 0.3 \\
-0.2 & 0.1
\end{bmatrix},\quad  A_3=A_4 = \begin{bmatrix}
0 & 0.3 \\
-0.2 & -0.1
\end{bmatrix}
\]
\[
B_1 =B_3= \begin{bmatrix}
0 & 0.1 \\
0 & 0.2
\end{bmatrix},\quad  B_2=B_4 = \begin{bmatrix}
0 & 0.1 \\
0 & 0
\end{bmatrix}
\]
In Example 1 of \cite{automatica_switched}, this system was shown to be exponentially stable for a delay bound of $L=13$.

The set $M_0$ of lifted transition matrices consists of the following four matrices:
\[
\widetilde{A}_1 = \begin{bmatrix}
0 & 0.3 & 0 & 0.1\\
-0.2 & 0.1 & 0 & 0.2 \\
1&0&0&0 \\
0&1&0&0
\end{bmatrix}
\quad
\widetilde{A}_2 = \begin{bmatrix}
0 & 0.3 & 0 & 0.1\\
-0.2 & 0.1 & 0 & 0 \\
1&0&0&0 \\
0&1&0&0
\end{bmatrix}
\]
\[
\widetilde{A}_3 = \begin{bmatrix}
0 & 0.3 & 0 & 0.1\\
-0.2 & -0.1 & 0 & 0.2 \\
1&0&0&0 \\
0&1&0&0
\end{bmatrix}
\quad
\widetilde{A}_4 = \begin{bmatrix}
0 & 0.3 & 0 & 0.1\\
-0.2 & -0.1 & 0 & 0 \\
1&0&0&0 \\
0&1&0&0
\end{bmatrix}
\]
Note that here $M_0=RI(M_0)$.
Furthermore, $\rho(\left\vert\widetilde{A}_1\right\vert)=\rho(\left\vert\widetilde{A}_3\right\vert)\approx 0.59<1$ and $\rho(\left\vert\widetilde{A}_2\right\vert)=\rho(\left\vert\widetilde{A}_4\right\vert)\approx 0.39<1$.
Hence, this switched system is intrinsically stable, so is in fact stable for any bound $L<\infty$, no matter the sequence of $A_{\sigma(k)}$ and $B_{\sigma(k)}$ that are chosen.
It is worth mentioning that this conclusion was reached without having to iteratively solve a system of linear matrix inequalities as in \cite{automatica_switched}.
\end{example}

In Example \ref{ex:5.1}, the set $M_0$ trivially satisfied $M_0=RI(M_0)$. In our next example, we consider the more nuanced case where $M_0$ is a proper subset of $RI(M_0)$.

\begin{example}{\textbf{\emph{(Row-Independent Closure of a Switched Network)}}}
Consider the following modified version of system (\ref{unliftedSwitched}), where we hold $A$ and $B$ constant in time, but add in a switching control term $\mathbf{c}_{\sigma(k)}\mathbf{u}^{k}$:
\begin{equation}\label{unliftedControl}
\mathbf{x}^{k+1} = A\mathbf{x}^k + B\mathbf{x}^{k-\tau(k)} + \mathbf{c}_{\sigma(k)}\mathbf{u}^{k} \quad \text{where}\quad 1\le \tau(k) \le L\quad \text{and}\quad   \mathbf{u}^{k} = \mathbf{q}^T\mathbf{x}^k.
\end{equation}
This system has the minimally-delayed lifted representation
\begin{equation}\label{LiftedControl}
\widetilde{\mathbf{x}}^{k+1} =
\widetilde{A}_{\sigma(k)}\widetilde{\mathbf{x}}^k\quad \text{where}\quad
\widetilde{A}_{\sigma(k)} = \begin{bmatrix}
A+\mathbf{c}_{\sigma(k)}\mathbf{q}^T & B \\
I_{n\times n} & \mathbf{0}_{n\times n}
\end{bmatrix}\quad \text{and}\quad
\widetilde{\mathbf{x}}^k=\begin{bmatrix}
\mathbf{x}^{k} \\
\mathbf{x}^{k-1}
\end{bmatrix}.
\end{equation}

In Example 2 of \cite{automatica_switched}, system \ref{unliftedControl} was shown to be stable with a upper delay bound $L=2$ for
\[
A = \begin{bmatrix}
0.7 & 0 \\
0.05 & 0.8
\end{bmatrix},\quad B = \begin{bmatrix}
-0.1 & 0 \\
-0.3 & -0.1
\end{bmatrix}, \quad \mathbf{q} = \begin{bmatrix}0.1510 \\ -0.2176\end{bmatrix}
\]
and
\[
\mathbf{c}_1 = \begin{bmatrix}0 \\ 0.01\end{bmatrix}, \quad
\mathbf{c}_2 = \begin{bmatrix}0.01 \\ 0\end{bmatrix}, \quad
\mathbf{c}_3 = \begin{bmatrix}0.01 \\ 0.01\end{bmatrix}.
\]

In this case we have three matrices in $M_0$:
\[
\widetilde{A}_1 = \begin{bmatrix}
0.7 & 0 & -0.1 & 0 \\
0.050151 & 0.7997824 & -0.3 & -0.1 \\
1&0&0&0 \\
0&1&0&0
\end{bmatrix}
\quad
\widetilde{A}_2 = \begin{bmatrix}
0.700151 & -0.0002176 & -0.1 & 0 \\
0.05 & 0.8 & -0.3 & -0.1 \\
1&0&0&0 \\
0&1&0&0
\end{bmatrix}
\]
\[
\widetilde{A}_3 = \begin{bmatrix}
0.700151 & -0.0002176 & -0.1 & 0 \\
0.050151 & 0.7997824 & -0.3 & -0.1 \\
1&0&0&0 \\
0&1&0&0
\end{bmatrix}.
\]
However, $RI(M_0)$ contains $\widetilde{A}_1,\widetilde{A}_2,\widetilde{A}_3$ as well as the following fourth matrix
\[
\widetilde{A}_4 = \begin{bmatrix}
0.7 & 0 & -0.1 & 0 \\
0.05 & 0.8 & -0.3 & -0.1 \\
1&0&0&0 \\
0&1&0&0
\end{bmatrix} \quad \text{corresponding to}\quad \mathbf{c}_4 = \begin{bmatrix}0 \\ 0\end{bmatrix}.
\]

Here $\rho(\left\vert\widetilde{A}_1\right\vert)\approx 0.9097<1$, $\rho(\left\vert\widetilde{A}_2\right\vert)\approx 0.9106<1$, $\rho(\left\vert\widetilde{A}_1\right\vert)\approx 0.9104<1$, $\rho(\left\vert\widetilde{A}_1\right\vert)\approx 0.9099<1$.
Thus this switched network is intrinsically stable. Therefore, using Main Result \ref{RIIntrinsic} we have been able to efficiently show that this system is stable for arbitrarily large delay bounds $L$.
\end{example}

It is worth emphasizing that while each of the examples in this section considers linear switched networks, the analysis applies directly to nonlinear switched networks as well, once the stability set $M_0$ of the various nonlinear mappings is known.

\section{Analytical and Computational Considerations}\label{sec:6}
We now summarize how these results can be used to analyze dynamical systems with a network structure. That is, for a dynamical network with either constant time-delays $(F_D,X_L)$ or network with time-varying time-delays with instances given by $(\{F_{D^{(k)}}\},X_L)$ using the following steps:\\

\emph{Step 1:} Consider the simpler undelayed version of the network $(F,X)$ where all interactions occur instantaneously. For a system given by Equation \eqref{unlifted} consider the minimally delayed version of the system given by Equation \eqref{UnliftedUndelayed}.\\

\emph{Step 2:} Find a Lipschitz matrix $A$ of the network by Equation \eqref{StabilityMatrix} if the mapping $F$ is piecewise differentiable and $X=\mathbb{R}^n$, otherwise directly determine the Lipschitz constants $A=[a_{ij}]$ by use of Definition \ref{Lipschitz Matrix}. Recall that there are infinitely many Lipschitz matrices $A$ for a given dynamical network. Ideally we would like to find one which minimizes the spectral radius $\rho(A)$, as this improves our estimate of the convergence rate of the system to its unique equilibrium, if $\rho(A)<1$. However, in certain cases it may simplify analysis considerably to merely find a bound for each entry of the Lipschitz matrix. Specifically, a bound that shows $\rho(A)<1$ for some Lipshitz matrix $A$.\\

\emph{Step 3:}
The spectral radius $\rho(A)$ can be computed efficiently in $\mathcal{O}(mn)$ time by use of the power method, where $m$ is the number of nonzero entries in $A$ \cite{powermethod}. As soon as a single Lipschitz matrix $A$ of $(F,X)$ satisfies $\rho(A)<1$, even if this $A$ does not have minimal spectral radius, the results of Theorem \ref{ConstantIntrinsic} and Main Result \ref{TimeVaryingIntrinsic} apply. These results guarantee that all delayed versions of $(F,X)$ are stable, even if the delays are varying in time so long as the magnitude of these delays are eventually bounded by some $L<\infty$. Moreover, the convergence rate of the delayed system is given by $\rho(A_L)$, where $A_L$ is defined in Main Result \ref{TimeVaryingIntrinsic}. Since $A_L$ is sparse with $m+nL$ nonzero entries, $\rho(A_L)$ can also be computed in $\mathcal{O}(mnL+n^2L^2)$ time using the power method.

\section{Conclusion}\label{sec:8}
The method described in this paper for determining whether a network is stable or can be destabilized by time-delays has a number of advantages over other methods. First, one need not consider the system itself but rather the simpler undelayed, and therefore lower-dimensional, version of the system. Second, the method(s) described here, at least for networks, require only the computation and spectral analysis of a single matrix rather than the use of Lyapunov-type methods, Linear Matrix Inequalities, or Semi-Definite Programming methods. Hence, very large systems can be analyzed using this method under the condition that their Lipschitz matrix can be efficiently computed. Moreover, if a system (process) can be shown (designed) to be intrinsically stable, there is no need to formally include delays in the system (model). The reason is that delays do change the qualitative dynamics of the system and it will same asymptotic state whether or not its delays are included. As has been shown, this is not the case for general dynamical networks.

One question that remains open is if a system is not intrinsically stable does there exist a time-delayed version of the system that is unstable. Another is that certain systems may be only \emph{locally intrinsically stable}, meaning that delaying certain network interactions may have no effect on the network's stability while delaying other interactions may change the system's dynamics. Determining for a given nonintrinsically stable network which is which would be important for determining which parts of the network are susceptible to this specific type of attack.

Lastly, as mentioned in the introduction networks are dynamic in two distinct ways. The first is the one considered in this paper, which is the changing state of the network's elements. The seconding is the evolving structure or \emph{topology} of the network. As time-varying time-delays effect the network's structure of interactions these delays also effect the underlying topology of the network. An important implication of this paper is that certain topological changes to a network, e.g. those induced by time-delays, can in general have a destabilizing effect on the network's dynamics. However, if the network's dynamics are intrinsically stable, this class of topological transformations does not effect the network's asymptotic state. It is unknown if there are other types of intrinsic dynamics, i.e. other stronger forms of dynamics, that are resilient to changes in the network's structure.

\section{Appendix}

Here we give the proofs of the results found in this paper. We begin by proving the result(s) of Sections \ref{sec:3}.

\subsection{Appendix A}

A proof of Proposition \ref{FixedPoint} is the following.

\begin{proof}
Let $\mathbf{x}^* = [x_1^*,x_2^*,\hdots,x_n^*]^T$ be a fixed point of a dynamical network $(F,X)$.
Then by definition
\[
F_i(x_1^*,x_2^*,\hdots,x_n^*) = x_i^* \quad \emph{ for all } \quad 1 \le i \le n.
\]
Let $D$ be a delay distribution with $\max_{ij}d_{ij}\le L$ for some finite $L>0$.
With the usual ordering of the component spaces of $\mathbf{x}\in X_L$, we have
\[
F_D(E_L(\mathbf{x}^*)) = \begin{bmatrix}
(F_D)_{1,0}(x_1^*,x_2^*,\hdots,x_n^*) \\
(F_D)_{2,0}(x_1^*,x_2^*,\hdots,x_n^*) \\
\vdots \\
(F_D)_{n,0}(x_1^*,x_2^*,\hdots,x_n^*) \\
(F_D)_{1,1}(x_1^*) \\
(F_D)_{2,1}(x_2^*) \\
\vdots \\
(F_D)_{n,L}(x_n^*)
\end{bmatrix} = \begin{bmatrix}
F_1(x_1^*,x_2^*,\hdots,x_n^*) \\
F_2(x_1^*,x_2^*,\hdots,x_n^*) \\
\vdots \\
F_n(x_1^*,x_2^*,\hdots,x_n^*) \\
x_1^* \\
x_2^* \\
\vdots \\
x_n^*
\end{bmatrix} = \begin{bmatrix}
x_1^* \\
x_2^* \\
\vdots \\
x_n^* \\
x_1^* \\
x_2^* \\
\vdots \\
x_n^* \\
\end{bmatrix} = E_L(\mathbf{x}^*).
\qedhere
\]
Hence, the extended fixed point $E_L(x^*)$ is a fixed point of $(F_D,X_L)$.
\end{proof}

\subsection{Appendix B}

Next we give a proof of Main Result \ref{LimitOrbits}, one of the two main results found in Section \ref{sec:4}.

\begin{proof}
Let $\mathbf{x},\mathbf{y} \in X$ and $k > 0$ be arbitrary.
Note that for any $F\in M$ with corresponding $A\in S$ we have, by definition of $A$ being a Lipschitz matrix of $F$, that
\[
\begin{bmatrix}
d_1(F_1(\mathbf{x}),F_1(\mathbf{y})) \\
\vdots \\
d_n(F_n(\mathbf{x}),F_n(\mathbf{y}))
\end{bmatrix} \preceq A
\begin{bmatrix}
d_1(x_1,y_1) \\
\vdots \\
d_n(x_n,y_n)
\end{bmatrix}
\]
where $\preceq$ denotes an element-wise inequality.
Thus, for the specific instance $(\{F^{(k)}\}_{k=1}^\infty,X)$ given in the hypothesis, we have inductively

\begin{align*}
\begin{bmatrix}
d_1(\mathscr{F}_1^k(\mathbf{x}),\mathscr{F}_1^k(\mathbf{y})) \\
\vdots \\
d_n(\mathscr{F}_n^k(\mathbf{x}),\mathscr{F}_n^k(\mathbf{y}))
\end{bmatrix} &=
\begin{bmatrix}
d_1(F_1^{(k)} \circ \mathscr{F}^{k-1}(\mathbf{x}),F_1^{(k)} \circ \mathscr{F}^{k-1}(\mathbf{y})) \\
\vdots \\
d_n(F_n^{(k)} \circ \mathscr{F}^{k-1}(\mathbf{x}),F_n^{(k)} \circ \mathscr{F}^{k-1}(\mathbf{y}))
\end{bmatrix} \\
&\preceq A^{(k)}
\begin{bmatrix}
d_1(\mathscr{F}_1^{k-1}(\mathbf{x}),\mathscr{F}_1^{k-1}(\mathbf{y})) \\
\vdots \\
d_n(\mathscr{F}_n^{k-1}(\mathbf{x}),\mathscr{F}_n^{k-1}(\mathbf{y}))
\end{bmatrix} \\
&\preceq A^{(k)} A^{(k-1)} \hdots A^{(1)}
\begin{bmatrix}
d_1(x_1,y_1) \\
\vdots \\
d_n(x_n,y_n)
\end{bmatrix}.
\end{align*}
By the definition of the joint spectral radius, there exists some positive constant $C$ (possibly dependant on $\mathbf{x}$ and $\mathbf{y}$) such that
\begin{align*}
d_{max}(\mathscr{F}^k(\mathbf{x}),\mathscr{F}^k(\mathbf{y})) &=
\left\Vert \begin{bmatrix}
d_1(\mathscr{F}_1^k(\mathbf{x}),\mathscr{F}_1^k(\mathbf{y})) \\
\vdots \\
d_n(\mathscr{F}_n^k(\mathbf{x}),\mathscr{F}_n^k(\mathbf{y}))
\end{bmatrix} \right \Vert_\infty \\
&\le
\left\Vert A^{(k)} A^{(k-1)} \hdots A^{(1)}
\begin{bmatrix}
d_1(x_1,y_1) \\
\vdots \\
d_n(x_n,y_n)
\end{bmatrix}
\right\Vert_\infty \\
&\le
C (\overline{\rho}(S))^k
\end{align*}
where $\left\Vert\mathbf{x}\right\Vert_\infty=\max_i\left\lvert x_i\right\rvert$.

Now, assume $\mathbf{x}^*$ is a shared fixed point of $(F,X)$ for all $F\in M$.
Then $\mathscr{F}^{1}(\mathbf{x}^*)=F^{(1)}(\mathbf{x}^*)=\mathbf{x}^*$, and if it is assumed that $\mathscr{F}^{k-1}(\mathbf{x}^*)=\mathbf{x}^*$, then it follows immediately that
\[
\mathscr{F}^{k}(\mathbf{x}^*) = F^{(k)} \circ \mathscr{F}^{k-1}(\mathbf{x}^*) = F^{(k)}(\mathbf{x}^*) = \mathbf{x}^*.
\]
Hence, by induction, $\mathbf{x}^*$ is a fixed point of $(\{F^{(k)}\}_{k=1}^\infty,X)$.
Thus
\[
d_{max}(\mathscr{F}^k(\mathbf{x}^0),\mathbf{x}^*)=d_{max}(\mathscr{F}^k(\mathbf{x}^0),\mathscr{F}^k(\mathbf{x}^*)) \le C\overline{\rho}(S)^k
\]
so $\lim_{k\to\infty}\mathscr{F}^k(\mathbf{x}^0)=\mathbf{x}^*$ for all initial conditions $\mathbf{x}^0\in X$.
\end{proof}

\subsection{Appendix C}

Next we prove Main Result \ref{RIIntrinsic}, which we do by dividing the proof of this result into several lemmata. First, we make explicit the effect of time delays on the Lipschitz matrix of a dynamical network.

\begin{lemma}{\textbf{\emph{(Structure of the Lipschitz Matrix of a Delayed Network)}}}\label{structure}
Let $(F,X)$ be a dynamical network with Lipschitz matrix $A=[a_{ij}] \in \mathbb{R}^{n\times n}$ and $D=[d_{ij}] \in \mathbb{N}^{n\times n}$ a delay distribution matrix with $\max_{i,j}d_{ij}\le L$.
Let $A_D$ be defined in terms of $A$ as
\[
A_D = \begin{bmatrix}
  A_0    &   A_1& \hdots & A_{L-1} &   A_L  \\
  I_n    &    0 & \hdots &    0    &    0   \\
  0      & I_n  & \hdots &    0    &    0   \\
  \vdots &\vdots& \ddots & \vdots  & \vdots \\
  0      &   0  & \hdots & I_n     &    0   \\
\end{bmatrix} \in \mathbb{R}^{n(L+1)\times n(L+1)}
\]
where each $A_\ell \in \mathbb{R}^{n\times n}$ is defined element-wise as
$A_{\ell} = \begin{bmatrix} a_{ij} 1_{d_{ij}=\ell} \end{bmatrix}$,
with the indicator function $1_{d_{ij}=\ell}$ defined as
\[
1_{d_{ij}=\ell} = \begin{cases}
1&\text{ if  }d_{ij}=\ell \\
0 & \text{ otherwise.}
\end{cases}\quad \text{ for } 0\le\ell\le L
\]
Then $A_D$ is a Lipschitz matrix of $(F_D,X_L)$.
\end{lemma}

\begin{proof}
Recall that for $\mathbf{x}\in X_L$, we order the components $x_{i,\ell}$ of $\mathbf{x}$ as
\[
\mathbf{x}=[x_{1,0},x_{2,0},\hdots,x_{n,0},x_{1,1},x_{2,1},\hdots,x_{n,L}]^T
\]
where $x_{i,\ell}\in X_{i,\ell}$ for $i=1,2,\hdots,n$ and $\ell=0,1,\hdots,L$.
Let $\mathbf{x},\mathbf{y}\in X$ be given.
Then
\begin{align*}
d_{i,0}\left((F_D)_{i,0})(\mathbf{x}),(F_D)_{i,0})(\mathbf{y})\right) &=
d_i\left(F_i(x_{1,d_{i1}},x_{2,d_{i2}},\hdots,x_{n,d_{in}}),F_i(y_{1,d_{i1}},y_{2,d_{i2}},\hdots,y_{n,d_{in}})\right) \\
&\le \sum_{j=1}^n a_{ij}d_j(x_{j,d_{ij}},y_{j,d_{ij}})
= \sum_{\ell=0}^L \sum_{j=1}^n a_{ij}1_{d_{ij}=\ell}d_{j,\ell}(x_{j,\ell},y_{j,\ell})
\end{align*}
which matches the first $n$ rows of $A_D$.
For $\ell\ge1$,
\[
d_{i,\ell}\left((F_D)_{i,\ell})(\mathbf{x}),(F_D)_{i,\ell})(\mathbf{y})\right)
= d_{i,\ell-1}(x_{i,\ell-1},y_{i,\ell-1}),
\]
which yields the identity matrices $I_n$ in $A_D$.
\end{proof}

The following theorem follows as a direct corollary of Lemma 3.3 of \cite{BunWebb2}, which is needed in our proof of Main Result \ref{RIIntrinsic}.

\begin{theorem}\label{bunwebbtheorem}
Let $(F,X)$ be a dynamical network with Lipschitz matrix $A$. Then for any delay distribution $D$ the constant time-delayed dynamical network $(F_D,X_D)$ has the Lipschitz matrix $A_D$ with\\
\indent (i) $\rho(A)\leq\rho(A_D)<1$ if $\rho(A)<1$;\\
\indent (ii) $\rho(A_D)=1$ if $\rho(A)=1$; and\\
\indent (iii) $\rho(A)\geq \rho(A_D)>1$ if $\rho(A)>1$.
\end{theorem}
We thus have the following immediate corollary by Theorem \ref{bunwebbtheorem} part (i):
\begin{theorem}
Suppose $D$ and $\hat{D}$ are delay distribution matrices such that  $D\preceq\hat{D}$, i.e. entries of $D$ are less than or equal to the corresponding entries of $\hat{D}$. If $(F,X)$, $(F_D,X_D)$, and $(F_{\hat{D}},X_{\hat{D}})$ have the corresponding Lipschitz matrices $A$, $A_D$, and $A_{\hat{D}}$, respectively, with $\rho(A)<1$ then
\[
\rho(A)\leq\rho(A_D)\leq\rho(A_{\hat{D}})<1.
\]
\end{theorem}
In other words, the spectral radius of the network is monotonic with respect to the addition of delays if $\rho(A)<1$.


We now require the following results regarding the joint spectral radius. First we need the following definition and theorem originally occurring as Equation (3.1) and Theorem 2 in \cite{JointSpectral} regarding sets of matrices with independent row uncertainties, respectively.

\begin{definition}{\textbf{\emph{(Independent Row Uncertainties)}}}
We say that a set of matrices $S\subset \mathbb{R}^{n \times n}$ has independent row uncertainty if $S$ can be expressed as
\[
S = \{(\mathbf{a}_1,\mathbf{a}_2,\hdots,\mathbf{a}_n)^T \text{ }| \text{ }\mathbf{a}_i \in Q_i, 1\le i\le n\}
\]
where the sets $Q_i\subset \mathbb{R}^n$, $1\le i\le n$ are closed and bounded.
\end{definition}

\begin{theorem}{\textbf{\emph{(Joint Spectral Radius of Nonnegative Matrices with Independent Row Uncertainty)}}}\label{JSRindpependent}
Let $S$ be a set of nonnegative matrices with independent row uncertainty. Then
\[
\overline{\rho}(S) = \max_{A\in S}\rho(A).
\]
\end{theorem}

Furthermore, we will use the equivalence of Definition \ref{JointSpectraldefinition} of the joint spectral radius with the following representation from \cite{JSRAlternate}.

\begin{theorem}{\textbf{\emph{(Alternate Form of the Joint Spectral Radius)}}}\label{AlternateJointSpectral}
Given a set of matrices $S\subset \mathbb{R}^{n\times n}$, the joint spectral radius $\rho(S)$ is given by
\[
\rho(S) = \limsup_{k\rightarrow\infty} \max \{||A||^{\frac{1}{k}}: A\text{ is a product of length }k\text{ of matrices in }S\}
\]
\end{theorem}

\noindent It follows immediately from Theorem \ref{AlternateJointSpectral} that if $S_1\subset S_2$, then $\rho(S_1)\le\rho(S_2)$. This allows us to give the following proof of Proposition \ref{RIradius}.

\begin{proof}
Let $(M,X)$ be a switched network with Lipschitz set $S$.
By the form of Theorem \ref{AlternateJointSpectral} we have $\overline{\rho}(S)\le \overline{\rho}(RI(S))$ since $S\subset RI(S)$.
For $1\le i\le n$, let $Q_i=\{\mathbf{a}_i|\text{ }A\in\overline{S}\}$ be the set of $i^{th}$ rows of all $A\in\overline{S}$.
Then $RI(S)$ may be expressed as
\[
S = \{(\mathbf{a}_1,\mathbf{a}_2,\hdots,\mathbf{a}_n)^T \text{ }| \text{ }\mathbf{a}_i \in Q_i, 1\le i\le n\}
\]
implying $RI(S)$ is row-independent.
Thus
\[\overline{\rho}(S)\le \overline{\rho}(RI(S)) = \max_{A\in RI(S)}\rho(A)
\]
by Theorem \ref{JSRindpependent}, as desired.
\end{proof}

\begin{lemma}{\textbf{\emph{(Equality of $A_D$'s)}}}\label{ADs}
Let $L>0$ and $1\le i\le n$ be given.
Let the matrices $A^{(1)}$, $A^{(2)}$ satisfy $\mathbf{a}^{(1)}_i=\mathbf{a}^{(2)}_i$, and the matrices $D^{(1)}$, $D^{(2)}$ satisfy $\mathbf{d}^{(1)}_i=\mathbf{d}^{(2)}_i$.
Then $(A^{(1)}_{D^{(1)}})_i=(A^{(2)}_{D^{(2)}})_i$.
\end{lemma}

\begin{proof}
It suffices to show that $(A^{(1)}_\ell)_i=(A^{(2)}_\ell)_i$, where given some $A$ and $D$, $A_\ell$ is defined as in Lemma \ref{structure}.
Let $0\le\ell\le L$ be arbitrary.
Then by Lemma \ref{structure},
\[
(A^{(1)}_\ell)_{ij}=\mathbf{a}^{(1)}_{ij}1_{d^{(1)}_{ij}=\ell}=\mathbf{a}^{(2)}_{ij}1_{d^{(2)}_{ij}=\ell}=(A^{(2)}_\ell)_{ij}\quad\text{for}\quad 1\le j\le n.
\]
Thus $(A^{(1)}_\ell)_i=(A^{(2)}_\ell)_i$, so by Lemma \ref{structure}, $(A^{(1)}_{D^{(1)}})_i=(A^{(2)}_{D^{(2)}})_i$.
\end{proof}

\begin{lemma}{\textbf{\emph{(Equality of sets)}}}\label{DelayThenRI}
Let $(M_0,X)$ be a switched network with Lipschitz set $S$.
Let $L>0$ and $\mathbb{D}=\{D\in\mathbb{N}^{n\times n}|\text{ }\max_{ij}d_{ij}\le L\}$.
Then $RI(\{A_D|\text{ }A\in S,D\in\mathbb{D}\})=\{A_D|\text{ }A\in RI(S),D\in\mathbb{D}\}$.
\end{lemma}

\begin{proof}
Let $(A_D)^*\in RI(\{A_D|\text{ }A\in S,D\in\mathbb{D}\})$.
Then there exist $(A_D)^{(1)},\hdots,(A_D)^{(n)}\in\{A_D|\text{ }A\in S,D\in\mathbb{D}\}$ such that $(A_D)^*_i=((A_D)^{(i)})_i$.
Furthermore, there exist $A^{(1)},\hdots,A^{(n)}\in S$ and $D^{(1)},\hdots,D^{(n)}\in \mathbb{D}$ such that $(A_D)^{(i)}=A^{(i)}_{D^{(i)}}$.
Let $A^*$ be constructed as $(A^*)_i=\mathbf{a}^{(i)}_i$, and $D^*$ be constructed as $(D^*)_i=\mathbf{d}^{(i)}_i$.
Then $A^*\in RI(S)$, and since each $D^{(i)}$ satisfies $d_{ij}\le L$, we have $D^*\in\mathbb{D}$.
Thus,
\[
(A_D)^*_i=((A_D)^{(i)})_i=(A^{(i)}_{D^{(i)}})_i=(A^*_{D^*})_i\quad\text{for}\quad 1\le i\le n,
\]
where the last equality follows from Lemma \ref{ADs}.
Therefore, $(A_D)^*=A^*_{D^*}\in\{A_D|\text{ }A\in RI(S),D\in\mathbb{D}\}$, so
\[
RI(\{A_D|\text{ }A\in S,D\in\mathbb{D}\})\subset\{A_D|\text{ }A\in RI(S),D\in\mathbb{D}\}.
\]

Now let $A_D\in\{A_D|\text{ }A\in RI(S),D\in\mathbb{D}\}$.
Then there exist $A^{(1)},\hdots,A^{(n)}\in S$ and $D\in\mathbb{D}$ such that $(A_D)_i=((A^{(i)})_D)_i$.
Let $(A_D)^{(i)}=(A^{(i)})_D$.
Then $(A_D)^{(i)}\in\{A_D|\text{ }A\in S,D\in\mathbb{D}\}$ so $A_D\in RI(\{A_D|\text{ }A\in S,D\in\mathbb{D}\})$.
Hence,
\[
\{A_D|\text{ }A\in RI(S),D\in\mathbb{D}\}\subset RI(\{A_D|\text{ }A\in S,D\in\mathbb{D}\})
\]
completing the proof.
\end{proof}

We now give the proof of Main Result \ref{RIIntrinsic} found in Section \ref{sec:5}. Following this we show that Main Result \ref{TimeVaryingIntrinsic} is a corollary of this result.

\begin{proof}
Let $(M_0,X)$ be a switched network with Lipschitz set $S$.
Assume $\mathbf{x}^*$ is a shared fixed point of $(F,X)$ for all $F\in M_0$ and $\rho(A)<1$ for all $A\in RI(S)$.
Let $L>0$, $\mathbb{D}=\{D\in\mathbb{N}^{n\times n}|\text{ }\max_{ij}d_{ij}\le L\}$, $M_d=\{F_D|F\in M_0,D\in\mathbb{D}\}$, and let $S_d$ be the Lipschitz set of $M_d$.
We will show $\rho(A_D)<1$ for all $A_D\in RI(S_d)$ and invoke Proposition \ref{RIradius}.

By Lemma \ref{DelayThenRI}, we have $RI(S_d)=\{A_D|\text{ }A\in RI(S),D\in\mathbb{D}\}$.
Then
\[
\max_{A_D\in RI(S_d)}\rho(A_D)=\max_{A\in RI(S)}\max_{D\in\mathbb{D}}\rho(A_D)<1
\]
by the hypothesis and Theorem \ref{bunwebbtheorem}.
Now given some $A\in RI(S)$, by Lemma \ref{structure} and repeated application of Theorem \ref{bunwebbtheorem} we have that
\[
\max_{D\in\mathbb{D}}\rho(A_D)=\rho(A_L)\quad\text{where}\quad
A_L = \begin{bmatrix}
\mathbf{0}_{n\times nL} & A \\
\mathbf{I}_{nL\times nL} & \mathbf{0}_{n\times n}
\end{bmatrix}.
\]
Thus by Proposition \ref{RIradius},
\[
\overline{\rho}(S_d)\le\overline{\rho}(RI(S_d))=\max_{A_D\in RI(S_d)}\rho(A_D)=\max_{A\in RI(S)}\rho(A_L)<1.
\]

Since $\mathbf{x}^*$ is a shared fixed point of $(F,X)$ for all $F\in M_0$, by Proposition \ref{FixedPoint} $E_L(\mathbf{x}^*)$ is a shared fixed point of $(F_D,X_L)$ for all $F_D\in M_d$.
Thus by Main Result \ref{LimitOrbits}, $E_L(\mathbf{x}^*)$ is a globally attracting fixed point of every instance $(\{F^{(k)}_{D^{(k)}}\}_{k=1}^\infty,X_L)$ of $(M_d,X_L)$.
\end{proof}

Lastly, we give a proof of Main Result \ref{TimeVaryingIntrinsic}.

\begin{proof}
Let $M_0$ be the singleton set consisting of $F$.
Then the Lipschitz set $S$ of $M_0$ consists only of the matrix $A$, and so trivially satisfies $S=RI(S)$.
Thus the hypothesis of Main Result \ref{RIIntrinsic} is trivially satisfied, and the result follows.
\end{proof}

\bibliographystyle{unsrt}

\section{References}
\begin{thebibliography}{9}

\bibitem{networks} 
Newman,M.E.J.
\emph{Networks an Introduction.}
Oxford University Press, Oxford (2010)

\bibitem{PeriodicBiologyDelays} 
an der Heiden, U. J.
Delays in physiological systems.
Mathematical Biology (1979) 8: 345. https://doi.org/10.1007/BF00275831

\bibitem{SyncPowerGrid} 
Nishikawa, Takashi, Ferenc Molnar, and Adilson E. Motter. "Stability Landscape of Power-Grid Synchronization." IFAC PapersOnLine 48, no. 18 (2015): 1-6, \\\texttt{https://www-sciencedirect-com.erl.lib.byu.edu/science/article/pii/S2405896315022594.}

\bibitem{MultistableGenes} 
Hasty, Jeff and McMillen, David and Isaacs, Farren and Collins, James J.
"Computational studies of gene regulatory networks: in numero molecular biology."
Nature Reviews Genetics 2, pp. 268-279 (2001). \\\texttt{
https://doi.org/10.1038/35066056}

\bibitem{IntroPaper} 
Sipahi, R., S. Niculescu, Chaouki T. Abdallah, W. Michiels, and Keqin Gu.
Stability and Stabilization of Systems with Time Delay.
IEEE Control Systems, Feb, 2011. 38.

\bibitem{destabilizing1} 
H. Logemann and S. Townley, ``The effect of small delays in the feed- back loop on the stability of neutral systems," Syst. Contr. Lett., vol. 27, pp. 267-274, 1996.

\bibitem{destabilizing2} 
W. Michiels, K. Engelborghs, D. Roose, and D. Dochain, ``Sensitivity to infinitesimal delays in neutral equations," SIAM J. Contr. Optim., vol. 40, pp. 1134-1158, 2002.

\bibitem{memory} V. Peddinti, D. Povey, and S. Khudanpur, ``A time delay neural network architecture for efficient modeling of longtemporal contexts," INTERSPEECH, 2015.

\bibitem{BunWebb0}
Webb, Benjamin and Leonid Bunimovich, ``Restrictions and Stability of Time-Delayed Dynamical Networks" Nonlinearity 26(8) DOI: 10.1088/0951-7715/26/8/2131 2012.

\bibitem{BunWebb2} 
Webb, Benjamin and Leonid Bunimovich. Isospectral Transformations: A New Approach to Analyzing Multidimensional Systems and Networks. New York, NY: Springer, 2014.

\bibitem{powermethod} 
Trefethen, L.N., Bau, D., 1998. Numerical Linear Algebra. Philadelphia: Series in Applied Mathematics,
Vol. 11. SIAM.

\bibitem{PeriodicPopulation} 
Zhao XQ. (2017) A Population Model with Periodic Delay. In: Dynamical Systems in Population Biology. CMS Books in Mathematics (Ouvrages de mathematiques de la SMC). Springer, Cham

\bibitem{StochasticTrains} 
Annabell Berger, Andreas Gebhardt, Matthias Muller-Hannemann, Martin Ostrowski.
Stochastic Delay Prediction in Large Train Networks.
In: 11th Workshop on Algorithmic Approaches for Transportation Modelling, Optimization, and Systems, pp. 100-111.
Schloss Dagstuhl--Leibniz-Zentrum fuer Informatik, 2011, \\\texttt{http://drops.dagstuhl.de/opus/volltexte/2011/3270} 

\bibitem{dependent1} 
Stojanovic, Sreten B., Dragutin L. J. Debeljkovic, and Nebojsa Dimitrijevic.
Stability of Discrete-Time Systems with Time-Varying Delay: Delay Decomposition Approach.
International Journal of Computers Communications and Control 7, no. 4 (Sep 16, 2014): 776.

\bibitem{dependent11} 
E.K. Boukas, Discrete-time systems with time-varying time delay: stability and stabilizability, \emph{Mathematical Problems in Engineering}, Article ID 42489:1-10, 2006.

\bibitem{dependent12} 
X.G. Liu, R.R. Martin, M. Wu and M.L. Tang, Delay-dependent robust stabilization of discrete-time systems with time-varying delay, \emph{IEEE Proc.: Control Theory and Applications}, 153(6): 689-702, 2006.

\bibitem{dependent13} 
V. Leite and M. Miranda, Robust Stabilization of Discrete-Time Systems with Time-Varying Delay: An LMI Approach, \emph{Mathematical Problems in Engineering}, 2008: Article ID 876509, 15 pages, 2008.

\bibitem{dependent14} 
K.F. Chen and I-K Fong, Stability of discrete-time uncertain systems with a time-varying state delay, Proc. IMechE, Part I: \emph{J. Systems and Control Engineering}, 222: 493-500, 2008.

\bibitem{neural3} 
Chena,S.,Zhaoa,W.,Xub,Y.
New criteria for globally exponential  stability of delayed Cohen-Grossberg neural network.
Math.Comput.Simul.79, 1527-1543  (2009)

\bibitem{neural5} 
Tao,L.,Ting,W.,Shumin,F.
Stability analysis on discrete-time Cohen-Grossberg neural networks with bounded  distributed delay.
In:Proceedings of the 30th Chinese Control Conference,Yantai,22-24 July  (2011)

\bibitem{neural1} 
Cohen,M.,Grossberg S.
Absolute stability and global pattern formation and parallel memory storage by competitive neural networks.
IEEE Trans.Syst.Man Cybern.SMC-13,  815-821 (1983)

\bibitem{neural2} 
Cao,J.
Global  asymptotic  stability of  delayed  bi-directional associative memory   neural networks.
Appl.Math.Comput.142(2-3),333-339  (2003)

\bibitem{neural4} 
Cheng,C.-Y.,Lin,K.-H.,Shih,C.-W.
Multistability in recurrent neural networks.
SIAM J.Appl.Math.66(4),1301-1320  (2006)

\bibitem{epidemic} 
Wang,L.,Dai,G.-Z.
Global stability of virus spreading in complex heterogeneous networks.
SIAMJ.Appl.Math.68(5),1495-1502 (2008)

\bibitem{ComputerNetworks} 
Alpcan,T.,Basar,T.
A globally stable adaptive congestion control scheme for internet-style networks with delay.
IEEE/ACMTrans.Netw.13, 6  (2005)

\bibitem{BunWebb1} 
Webb, Benjamin and Leonid Bunimovich.
Intrinsic Stability, Time Delays and Transformations of Dynamical Networks.
Advances in Dynamics, Patterns, Cognition. Springer International Publishing, 2017.

\bibitem{JointSpectral} 
Vincent D Blondel and Yurii Nesterov.
Polynomial-Time Computation of the Joint Spectral Radius for some Sets of Nonnegative Matrices.
SIAM Journal on Matrix Analysis and Applications 31, no. 3 (Sep 1, 2009): 865-876.

\bibitem{automatica_lifted} 
Zhang, Xian-Ming and Qing-Long Han.
Abel Lemma-Based Finite-Sum Inequality and its Application to Stability Analysis for Linear Discrete Time-Delay Systems.
Automatica 57, (Jul, 2015): 199-202, \\\texttt{https://www.sciencedirect.com/science/article/pii/S000510981500179X.}

\bibitem{dependent4} 
Li, Xu, Rui Wang, and Xudong Zhao.
Stability of Discrete-Time Systems with Time-Varying Delay Based on Switching Technique.
Journal of the Franklin Institute 355, no. 13 (Sep, 2018): 6026-6044, \\\texttt{https://www.sciencedirect.com/science/article/pii/S0016003218303910.}

\bibitem{dependent2} 
Fridman, Emilia.
Introduction to Time-Delay Systems : Analysis and Control, Section 6. Systems and Control: Foundations and Applications. 2014th ed.
Cham: Birkhauser Boston, 2014.

\bibitem{automatica_switched} 
Zhang, Wen-An and Yu, Li.
Stability analysis for discrete-time switched time-delay systems.
Automatica 45, (May, 2009): 2265-2271.

\bibitem{JSRAlternate} 
G.-C. Rota and W. G. Strang.
A note on the joint spectral radius.
Indag. Math., 22 (1960), pp. 379-381.

\bibitem{JSRHard}
Tsitsiklis, J.N. and Blondel.
``The Lyapunov exponent and joint spectral radius of pairs of matrices are hard-when not impossible-to compute and to approximate".
V.D. Math. Control Signal Systems (1997) 10: 31. https://doi.org/10.1007/BF01219774





\end{thebibliography}

\end{document}